\documentclass[12pt]{amsart}
\usepackage{graphicx,caption}
\usepackage{amssymb,mathtools,graphicx}
\usepackage{amssymb,amsmath,amsthm}
\usepackage{latexsym}
\usepackage{epsfig}
\usepackage[abs]{overpic}
\usepackage{color}
\usepackage{comment}
\usepackage[letterpaper,top=2cm,bottom=2cm,left=3cm,right=3cm,marginparwidth=1.75cm]{geometry}

\usepackage{amsmath}
\usepackage{graphicx}
\usepackage[colorlinks=true, allcolors=blue]{hyperref}
 \usepackage{float}
 
\newtheorem{theorem}{Theorem}[section]

\newtheorem{definition}[theorem]{Definition}
\newtheorem{example}[theorem]{Example}
\newtheorem{proposition}[theorem]{Proposition}
\newtheorem{remark}[theorem]{Remark}

\newtheorem{conjecture}[theorem]{Conjecture}
\newtheorem{problem}[theorem]{Problem}

\newtheorem*{theorem*}{Main Theorem}

\title{Path to homology of Yang-Baxter operators}
\author{Jozef H. Przytycki}
\begin{document}

\keywords{knot, distributive homology,  quandle, skein module, Yang-Baxter homology, Yang-Baxter operator, history of knot theory.\\
Mathclass: Primary 57K10; Secondary 16T25, 57K31, 57-03.\\
Published in LOOPS23: Nonassociative Algebra {\it Banach Center Publications},
Volume 129, Warszawa 2025, 133-182.
}
%\mathclass{Primary 57K10; Secondary 16T25, 57K31, 57-03.}

\title{Path to homology of Yang-Baxter operators}

%\author{J{o}zef H. Przytycki}
\address{Department of Mathematics, The George Washington University, Washington DC 20052, USA,  \\
and University of Gdansk, Poland,\\
ORCID: 0000-0002-1600-8889\quad\ E-mail: przytyck@gwu.edu}

\maketitle

\begin{abstract}

This paper is an extended version of two talks I gave during workshop ``Loops'13" in Bedlewo in June 2023. In the first talk I gave a historical introduction to Knot Theory. In the second, I traced my journey toward Yang-Baxter homology and this talk has a partially survey and a partially novel character.
\end{abstract}

\tableofcontents

\section{Introduction}
This paper is an extended version of two talks I gave at workshop ``Loops'13" in B{\c e}dlewo in June 2023. It has partially survey and partially novel character. 
In the first talk I gave a historical introduction to Knot Theory starting from G.W.Leibniz (1646-1716) ``Geometria Situs" and mentioning two mathematicians/philosophers who infuenced Leibniz greatly: Ramon Llull\footnote{Invited by my friend Marithania Silvero, I visited the University of Barcelona, in July, 2016. We were working on Khovanov homology \cite{PrSi1,PrSi2} (see Section 7.1 on my dream of connecting Khovanov homology and Yang-Baxter homology). I could realize then that Ramon Llull is kept there in great esteem, including university named after him.} (1232-1313) and Athanasius Kircher\footnote{Yes, it is the same Kircher who, as also Olga Tokarczuk mentions, influenced  Benedykt Chmielowski (1700-1763) in his Encyclopedia, \cite{Tok}.} (1601-1680). In the first part of the second talk, Section \ref{Attraction}, I desribe  ``Attraction of nonassociative structures" and how my work with Pawe{\l} Traczyk on generalizations of the Jones polynomial \cite{Jon2,PT1} led us to entropic structures; also that our work on HOMFLYPT polynomial led me to Skein Modules \cite{Prz2}. The substantial fragment of this part (see Subsections~\ref{3.5}--\ref{Mietek}) is devoted to $n$-moves on links and interesting formulas (in the commutative and distributive cases). Generally, the second part describes the winding path that led me to Yang-Baxter homology. The path includes many digressions and ends with a few open problems on which I am very eager to work.
In particular, I describe a mysterious decomposition of the third Reidemeister move into face maps, as shown in Figure~\ref{Y-B-2-cycle}. This decomposition still awaits further exploration and explanation. I came up with this idea while traveling by train from Gda\'{n}sk to Pozna\'{n} in May 2015 (see Section \ref{Section 6}).

\section{Short history of knot theory}

Knots have fascinated people from the dawn of the human history.
We can wonder what caused a merchant living about 1700 BC. in Anatolia
and exchanging goods with Mesopotamians, to choose braids and knots as
his seal sign; Figure \ref{AnatolianStampRa}. We can guess however that
stamps, cylinders and seals with knots
and links as their motifs appeared before proper writing was invented
about 3500 BC.
\ \\
\ \\
\begin{figure}
\centering
\scalebox{1.19}{\includegraphics{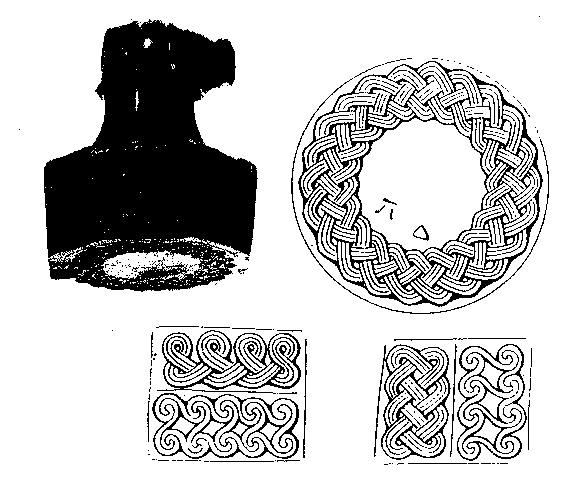}}
\caption{Stamp seal, about 1700 BC (the British Museum)\\
On the octagonal base [of hammer-handled haematite seal] are patterns surrounding a hieroglyphic inscription (largely erased). Four of the sides are blank and the other four are engraved with elaborate
patterns typical of the period (and also popular in Syria) alternating
with cult scenes...(\cite{Col}, p.93). }
\label{AnatolianStampRa}
\end{figure}

An even older example, shown in Figure~\ref{SumerianKnotRa}, is a cylinder seal impression (2600-2500 B.C.) from the Sumerian city-state of Ur in ancient Mesopotamia. It can be found in the book ``Innana" by Diane Wolkstein and Samuel Noah Kramer (page 7 of \cite{Wo-Kr}) to illustrate the following text:

``Then a serpent who could not be charmed
made its nest in the roots of the tree."
\ \\
%\vspace{3in}
\begin{figure}
\centering
\scalebox{1.8}{\includegraphics{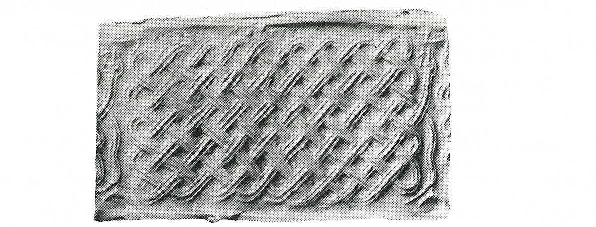}}
\caption{Snake with Interlacing Coil. Cylinder seal. Ur, Mesopotamia. The Royal Cemetery, Early Dynastic
period, c. 2600-2500 B.C. Lapis lazuli. Iraq Museum. Photograph
courtesy of the British Museum, UI 9080,
\cite{Wo-Kr} }
\label{SumerianKnotRa}
\end{figure}
See \cite{Wo-Kr}\footnote{On the pages 179-180 they comment:\ 
 The majority of the pictorial surface is covered with the inter-
twined coils of a serpent, forming a lattice pattern. To the right its tail
appears below the coils and its head above, with a bird perched upon it.
Two snakes intertwined rather than one are shown on earlier
representations of this motif. Snakes twist themselves together in this
fashion when mating, suggesting this symbol's association with fertility.}

 It is tempting to look for the origin of knot theory in Ancient
Greek mathematics (if not earlier). There is some justification to do so:
a Greek physician named Heraklas, who lived during the first century A.D.
and who was likely a pupil or associate of Heliodorus, wrote an essay
on surgeon's slings\footnote{Heliodorus, who lived at the time
of Trajan (Roman Emperor 98--117 A.D.), also mentions in
his work knots and loops \cite{Sar1}}\footnote{Hippocrates of
Cos (c.460 - 375 B.C.) in his collection of notes:
In the surgery; {\it De officina medici; Cat' i\={e}treion}, deals with
bandaging. Thessalos, Hippocrates' son, has been named also as the author.
 A commentary on the Hippocratic treatise on {\it Joints} was written by
 Apollonios of Citon (in Cypros), who flourished in Alexandria in the
first half of the first century B.C. That commentary has obtained a
great importance because of an accident in its transmition. A manuscript
 of it in Florence (Codex Laurentianus) is a Byzantine copy of the ninth
century, including surgical illustrations (for example, with reference to
 reposition methods), which might go back to the time of Apollonios and
even Hippocrates. Iconographic tradition of this kind are very rare,
because the copying of figures was far more difficult than the writing
of the text and was often abandoned \cite{Sar1} (page 365).
%G.Sarton, Ancient Science through the Golden Age of Greece, Dover Pub., 1993
%(first edition: Harvard Univ. Press, 1952).
The story of the illustrations to Apollonios' commentary is described in
\cite{Sar2}.}.
%G.Sarton, Introduction to the history of science, Baltimore: Williams
%and Wilkins, 1927-1948; (vol. 1 p.216).
Heraklas explains,  giving step-by-step instructions,
eighteen ways to tie orthopedic slings. His work survived because
Oribasius of Pergamum (ca. 325-400; physician of the emperor
 Julian the Apostate) included it toward the end of the fourth century
in his ``Medical Collections". The oldest extant manuscript of ``Medical
Collections" was made in the tenth century by the Byzantine physician
Nicetas. The Codex of Nicetas was brought to Italy in the fifteenth century by
an eminent Greek scholar,   J.~Lascaris,  a refugee from Constantinople.
Heraklas' part of the Codex of Nicetas has no illustrations,
and around 1500 an anonymous artist depicted Heraklas' knots in one of
the Greek manuscripts of Oribasus ``Medical Collections" (in Figure \ref{crossed} we
reproduce the drawing of the third Heraklas knot together with its
original,   Heraklas',  description). Vidus Vidius (1500-1569),
a Florentine who became physician to Francis I (king of France, 1515-1547)
 and professor of medicine in the Coll\`ege
de France,   translated the Codex of Nicetas into Latin; it contains also
drawings of Heraklas' surgeon's slings by the Italian painter,
sculptor and architect Francesco Primaticcio (1504-1570); \cite{Da,Ra}.
\begin{figure}
\centering
\scalebox{2.35}{\includegraphics{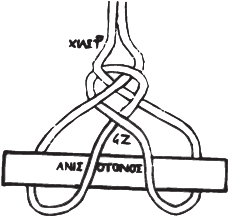}}
\caption{The crossed noose;\\
``For the tying the crossed noose,  a cord,
folded double,  is procured,
and the ends of the cord are held in the left hand,   and the loop is
held in the right hand. Then the loop is twisted so that the slack
parts of the cord crossed. Hence the noose is called crossed. After
the slack parts of the cord have been crossed,  the loop is placed
on the crossing,   and the lower slack part of the cord is pulled up
through the middle of the loop. Thus the knot of the noose is in the
middle,   with a loop on one side and two ends on the other. This
likewise,   in function,  is a noose of unequal tension";
\cite{Da}}
\label{crossed}
\end{figure}
\vspace*{1in}

\subsection{From Leibniz ``Geometria Situs" to Gauss}
Heraklas' essay should be taken seriously as far as knot theory is
concerned even if it is not knot theory proper but rather its application.
The story of the survival of Heraklas' work; and efforts to reconstruct
his knots in Renaissance is typical of all science disciplines and
efforts to recover lost Greek books provided the important engine for
development of modern science. This was true in Mathematics as well:\
the beginning of modern calculus in XVII century can be traced
to efforts of reconstructing lost books of Archimedes and other ancient
Greek mathematicians. It was only the work of Newton and Leibniz which
went much farther than their Greek predecessors.

There are two enlightening examples of great Renaissance artists
interest in knots:\
Engravings by Leonardo da Vinci\footnote{Giorgio Vasari writes
in \cite{Vas}: ``[Leonardo da Vinci] spent much time in making a regular design of a series of knots so that the cord may be traced from one end to the other, the whole filling a round space. There is a fine engraving of this most difficult design, and in the middle are the words: Leonardus Vinci Academia."} (1452-1519) \cite{Mac} (see Figure~\ref{Leonardo2Ra}) 
\begin{figure}[ht]
\centering
\scalebox{2.0}{\includegraphics{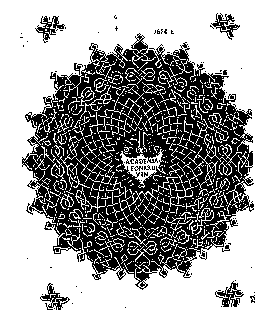}}
\caption{A knot by Leonardo \cite{Mac}; c. 1496 }
\label{Leonardo2Ra}
\end{figure}
and woodcuts by Albrecht
D\"urer\footnote{``Another great artist with whose works D\"urer now became acquainted was Leonardo da Vinci. It does not seem likely that the two artists ever met, but he may have been brought into relation with him through Luca Pacioli, the author of the book De Divina
Proportione, which appeared at Venice in 1509, and an intimate friend
of the great Leonardo. D\"urer would naturally be deeply interested in
the proportion theories of Leonardo and Pacioli. He was certainly
acquainted with some engravings of Leonardo's school, representing a
curious circle of concentric scrollwork on a black ground, one of them
entitled Accademia Leonardi Vinci; for he himself executed six woodcuts
in imitation, the Six Knots, as he calls them himself. D\"urer was
amused by and interested in all scientific or mathematical problems..."
From: http://www.cwru.edu/edocs/7/258.pdf, compare
\cite{Dur-2}.}\footnote{We can quote from the Tait fundamental paper {\it On Knots I}: ``I am indebted to Mr. Dallas for a photograph of a remarkable engraving by D\"urer, exhibiting a very complex but symmetrical linkage... "  } (1471-1528) \cite{Dur-1,Ha}, Figure \ref{Durer6knotsRa}.
\ \\
\begin{figure}
\centering
\scalebox{0.95}{\includegraphics{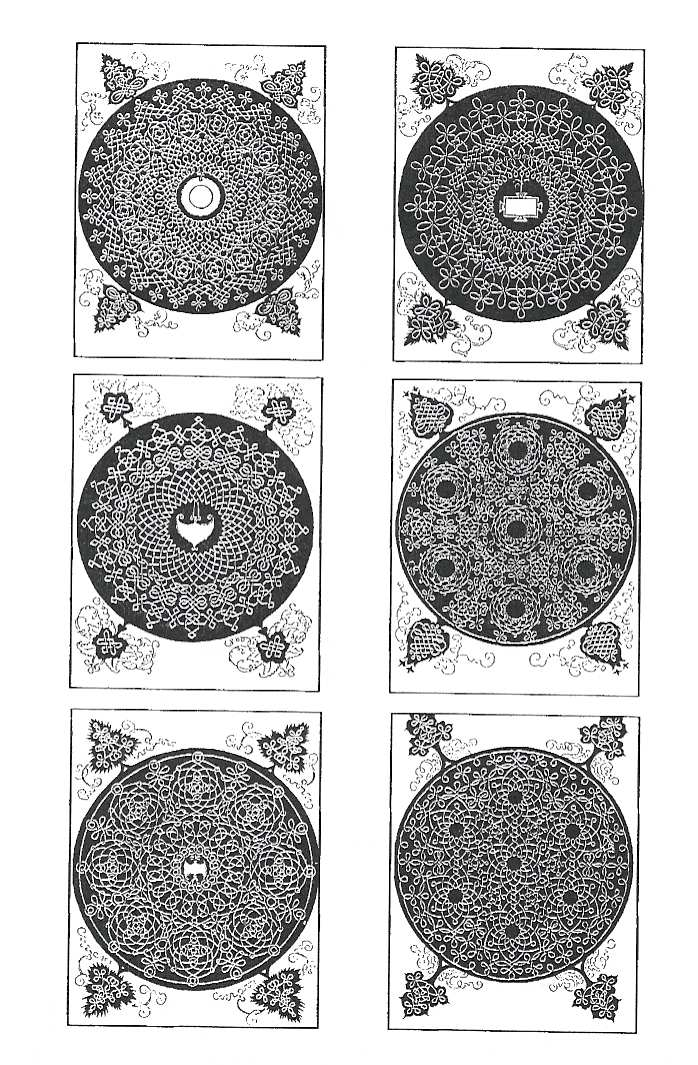}}
\caption{Six knots by D\"urer \cite{Kur}; c. 1505-1507}
\label{Durer6knotsRa}
\end{figure}
%\vspace{3.4in}
%\centerline{\psfig{figure=DurerKnot1.eps,height=11.5cm}}
%\centerline{\psfig{figure=Durer6knotsRa.eps,height=18.5cm}}
%\centerline{\psfig{figure=Durer6knots.eps,height=18.5cm}}
%\centerline{Fig. 2.3; \ Six knots by D\"urer \cite{Kur}; c. 1505-1507}
%\centerline{Fig. 2.3; \ A knot by D\"urer \cite{Kur}; c. 1505-1507}
\ \\

We would argue, that modern knot theory has its roots with
 Gottfried Wilhelm Leibniz (1646-1716) speculation that aside from calculus
and analytical geometry there should exist a ``geometry of position"
(geometria situs) which deals with relations depending on position
alone (ignoring magnitudes). In a letter to Christian Huygens (1629-1695),
written in 1679 \cite{Lei}, he declared:\\
``I am not content with algebra,  in that it yields neither the shortest
proofs nor the most beautiful constructions of geometry.
Consequently, in view of this, I consider that we need yet
another kind of analysis, geometric or linear, which deals directly
with position, as algebra deals with magnitude".

   I do not know whether Leibniz had any convincing example of a problem
belonging to the geometry of position. According to \cite{Kli}:\\
``As far back as 1679 Leibniz, in his Characteristica Geometrica,
tried to formulate basic geometric properties of geometrical figures,
to use special symbols to represent them,  and to combine these
properties under operations so as to produce others.
He called this study analysis situs or geometria
situs...  To the extent that he was at all clear, Leibniz envisioned what
we now call combinatorial topology".\footnote{One should also investigate whether
seeds of Leibniz Geometria Situs can be found in work of Ramon Llull (1232-1315) \cite{Bon}.
Also the influence of Athanasius Kircher (1602-1680) should be evaluated. In \cite{Fin,Gla} it is suggested
{\it that virtually every major scientific, linguistic, and historical project
on which he [Leibniz] embarked had been directly inspired by reading Kircher's works}. }.

The first convincing example
of geometria situs was Leonard Euler (1707-1783) solution to the bridges of K\"onigsberg
problem (1735).
 This concerns the bridges on the river Pregel
at K\"onigsberg (then in East Prussia).
\iffalse \footnote{Euler never
visited K\"onigsberg.
Euler was informed about the puzzle of bridges of K\"onigsberg
(and about a possible relation to Leibniz geometria situs) by
future mayor of Danzig (Gda\'nsk) Carl Leonhard
Gottlieb Ehler (1685-1753); there are 14 surviving letters from
 Ehler to Euler; the first is  dated April 8, 1735.
Ehler in turn was acting on behalf of Danzig mathematician
Heinrich Kuhn (1690-1769) \cite{H-W}. Kuhn was born in
K\"onigsberg, he studied at the Pedagogicum there,
... in 1733 he settled in Danzig. One should add that Kuhn
was the first person to suggest geometric interpretation of complex
numbers \cite{Jan,Kuhn}.}.
\fi
The problem was proposed by Heinrich K\"uhn (1690-1769) about 1735.
 Kuhn was a Danzig (Gda\'nsk) mathematician
 born in K\"onigsberg. He studied at the Pedagogicum there,
and in 1733 settled in Danzig as a mathematics professor at the Academic Gymnasium (he was also
a co-founder of the Nature Society and  the first person to suggest the geometric interpretation of complex
numbers \cite{Jan,Kuhn,Sz1,Sz2}). Kuhn communicated to Leonard Euler (1707-1783)
the puzzle of the bridges of K\"onigsberg, suggesting that this may
be an example of geometria situs. Kuhn was communicating,
in fact, through his friend Carl Leonhard Gottlieb Ehler (1685-1753), correspondent of Leibniz and
future mayor of Danzig.
The first extant\footnote{Possibly, the bridges of K\"onighsberg were mentioned in previous letters
of Ehler to Euler, or more likely, they discussed them when Ehler was in Petersburg.
Ehler met Euler in Petersburg in late 1734 or 1735 as a member of a delegation of Danzig to Empress of Russia,
asking for a reduction of reparations forced on Danzig by Russia in 1734 after the capitulation of Danzig
(the city was briefly occupied by the Russians after the prolonged Siege of Danzig during
the War of the Polish Succession (city capitulated June 30, 1734).
The city, which supported S. Leszczy\'nski, the losing candidate for
the Polish throne, was forced to pay reparations following the siege). {\color{red} Added for e-print: According to Ehler's son, Carl Ludwig Ehler, diary the delegation reached Petersburg October 5, 1734, compare \cite{Sz1}. } The delegation left Petersburg June 3, 1735;
\cite{Czer}. In view of this it is interesting to observe that Euler's paper has a note ``Based on a
talk presented to the Academy on 26 August 1735" just less than three months after Ehler left Petersburg. }
letter by Ehler concerning K\"onigsberg bridges is dated March 9, 1736. There
he writes: ``You would render to me and our friend K\"ohn a most valuable service, putting us greatly in your debt,
most learned Sir, if you would send us the solution, which you know well, to the
problem of the seven K\"onigsberg bridges, together with a proof.  It would prove to be an outstanding example
of {\bf Calculi Situs}, worthy of your great genius.  I have added a sketch of the said bridges \dots ''
In the reply of April 3, 1736 Euler writes ``\dots Thus you see,
most noble Sir, how this type of solution bears little relationship to
mathematics, and I do not understand why you expect a mathematician to produce it,
rather than anyone else, for the solution is based on reason alone, and its discovery does
not depend on any mathematical principle. Because of this,
I do not know why even questions which bear so little relationship to mathematics
are solved more quickly by mathematicians than by others. In the meantime,
most noble Sir, you have assigned this 
question to the {\bf geometry of position}, but I am ignorant as to what this new discipline
involves, and as to which types of problem Leibniz and Wolff expected to see expressed
in this way \dots '' \cite{H-W}.
However, when composing his famous paper on the bridges of K\"onigsberg, see Figure \ref{bridgesRa}, Euler
already agrees with K\"uhn's suggestion. The geometry of position figures even in the title of the
paper {\it Solutio problematis ad geometriam situs pertinentis}.\footnote{In the paper, Euler writes:
``The branch of geometry that deals with magnitudes has been zealously studied
throughout the past,  but there is another branch that has been almost
unknown up to now; Leibniz spoke of it first,  calling it the
``geometry of position'' (geometria situs).
This branch of geometry deals with relations dependent
on position; it does not take magnitudes into considerations, nor does it
involve calculation with quantities. But as yet no satisfactory definition
has been given of the problems that belong to this geometry of position
or of the method to be used in solving them. Hence, when a problem was recently mentioned,
which seemed geometrical but was so constructed that it did not require the measurement of
distances---especially as its solution involved only position, and no calculation
was of any use. I have therefore decided to give here the method which I have found for
solving this kind of problem, as an example of the geometry of position.\
2.\ The problem, which I am told is widely known, is as follows: in K\"onigsberg in Prussia,
there is\dots  ''\cite{Eu,BLW}.}

 Euler presented his solution (and generalization) to the bridges of K\"onigsberg problem
on August 26, 1735 to the Russian Academy at
St. Petersburg (it was submitted for publication in 1736), \cite{Eu}.
With the Euler paper, graph theory and topology were born.

%\vspace*{2.7in}
\begin{figure}
\centering
\scalebox{1.11}{\includegraphics{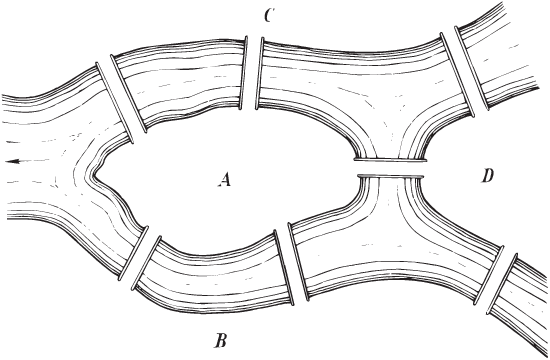}}
\caption{Seven bridges of K\"onigsberg}
\label{bridgesRa}
\end{figure}
%\begin{center}\centerline{\psfig{figure=bridgesRa.eps}}
%\centerline{\psfig{figure=bridges.eps}}Figure 3.1; Bridges of K\"onigsberg\end{center}

For the birth of knot theory one had to wait another 35 years. In 1771
Alexandre-Theophile Vandermonde (1735-1796) wrote the paper:
{\it Remarques sur les probl\`emes de situation} (Remarks on
problems of positions) where he specifically places braids and knots, compare Figure \ref{VanderRa}, as
a subject of the geometry of position \cite{Va}.\footnote{I had learnt about Vandermonde paper from the book {\it Graph theory 1736-1936},  \cite{BLW} and it was an incentive for me to write my first paper on history of knot theory {\it History of the knot theory from Vandermonde to Jones}, \cite{Prz3}, because no paper on (history of) knot theory was mentioning Vandermonde.}
In the first paragraph of the paper
Vandermonde wrote:

{\it Whatever the twists and turns of a system of threads
in space, one can always obtain an expression for the calculation of its
dimensions, but this expression will be of little use in practice.
The craftsman who fashions a braid, a net, or some knots will be concerned,
not with questions of measurement, but with those of position: what he sees
there is the manner in which the threads are interlaced.}

%\vspace*{2.5in}

\begin{figure}
\centering
\scalebox{1.11}{\includegraphics{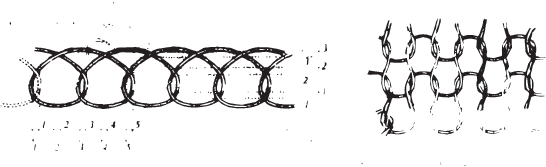}}
\caption{Knots of Vandermonde}
\label{VanderRa}
\end{figure}

%\centerline{\psfig{figure=VanderRa.eps}}
%\centerline{\psfig{figure=Vander.eps}}
%\begin{center} Figure 3.2; Knots of Vandermonde \end{center}

\subsection{Knot Theory from Gauss to modern knot theory}
   In our search for the origin of knot theory, we arrive next at
Carl Friedrich Gauss (1777-1855). According to  \cite{Stac,Dun}
%\footnote{The quotation is taken from the article by Paul St\"ackel:
%Gauss als Geometer, in the X'th Volume of Gauss' Collected Works. The
%quotation ends there: ``...den Gauss hat darauf vermerkt: `Riedl,
%{\em Beitr\"age zur Theorie des Sehnenwinkels}, Wien 1827'\ ". The note
%was translated to English by  Gauss' biographer \cite{Dun}}
:

``One of the oldest notes by Gauss to be found among his papers is a sheet
of paper with the date 1794. It bears the heading ``A collection of knots"
and contains thirteen neatly sketched views of knots with English names
written beside them... With it are two additional pieces of paper with
sketches of knots. One is dated 1819; the other is much later,  ...".
\footnote{According to \cite{Gr}, the first English sailing book with
pictures of knots appeared in 1769 \cite{Falc}.
%This book is illustrated with thirteen copper-plate engravings - we can
%speculate that these engravings were seen/copied by young Gauss.
}
In July of 1995 I finally visited the old library in G\"ottingen\footnote{I was invited to G\"ottingen by my Greek friend Sofia Lambropoulou, and the old library allowed me to look at all Gauss notebooks and even they made pictures for me of all drawings I asked them. These drawings are sitting now in some boxes in my GWU office.}, 
I looked at knots from 1794 - in fact not all of them are drawn - some
only described; see Figure \ref{Gauss-meshingknot} for one of the drawings\footnote{First
eight drawings are reproduced in the preface to \cite{T-G}.}.
\ \\
\begin{figure}
\centering
\scalebox{2.01}{\includegraphics{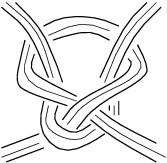}}
\caption{{\it Meshing knot}, 10'th knot of Gauss from 1794}
\label{Gauss-meshingknot}
\end{figure}
%\vspace*{4in}
%\centerline{\psfig{figure=Gauss-meshingknot.eps,height=5.1cm}}
%\begin{center}  Figure 3.3; {\it Meshing knot}, 10'th knot of Gauss from 1794.\end{center}

There are other fascinating drawings in Gauss' notebooks.
For example, the drawing of a braid with complex coordinate
description at each height (Figure \ref{Gauss-braid}; compare \cite{Ep, Prz4}),
 and the note that it is a good method of coding a knotting.
It is difficult to date the drawing; one can say for sure that it was done
between 1814 and 1830, I would guess closer to 1814\footnote{As a curiosity
one can add that one of the notebooks (Handb. 3) in which Gauss had
also drawn some knot diagrams has braids motives on its cover.}.
\ \\
%\vspace*{4in}
\begin{figure}
\centering
\scalebox{1.0}{\includegraphics{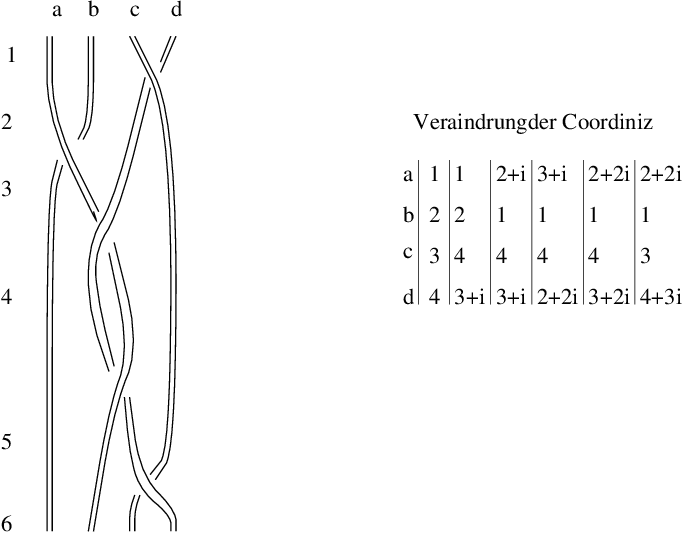}}
\caption{It is a good method of coding a knotting (from
a Gauss' notebook (Handb.7)). Gauss coordinates are not always
consistent; most of the time $i$ is pointing downward but there are exceptions}
\label{Gauss-braid}
\end{figure}

%\centerline{\psfig{figure=Gauss-braid.eps}}
%\begin{center} Figure \ref{Gauss-braid}\footnote{Gauss coordinates are not always
%consistent; most of the time $i$ is pointing downward but there are exceptions.} \end{center}

  There is also the mysterious ``framed tangle",
 see Figure \ref{gauss12Arx} \cite{Ga-1,Prz6} whose interpretation is not yet
convincingly given.
\begin{figure}
\centering
\scalebox{.45}{\includegraphics{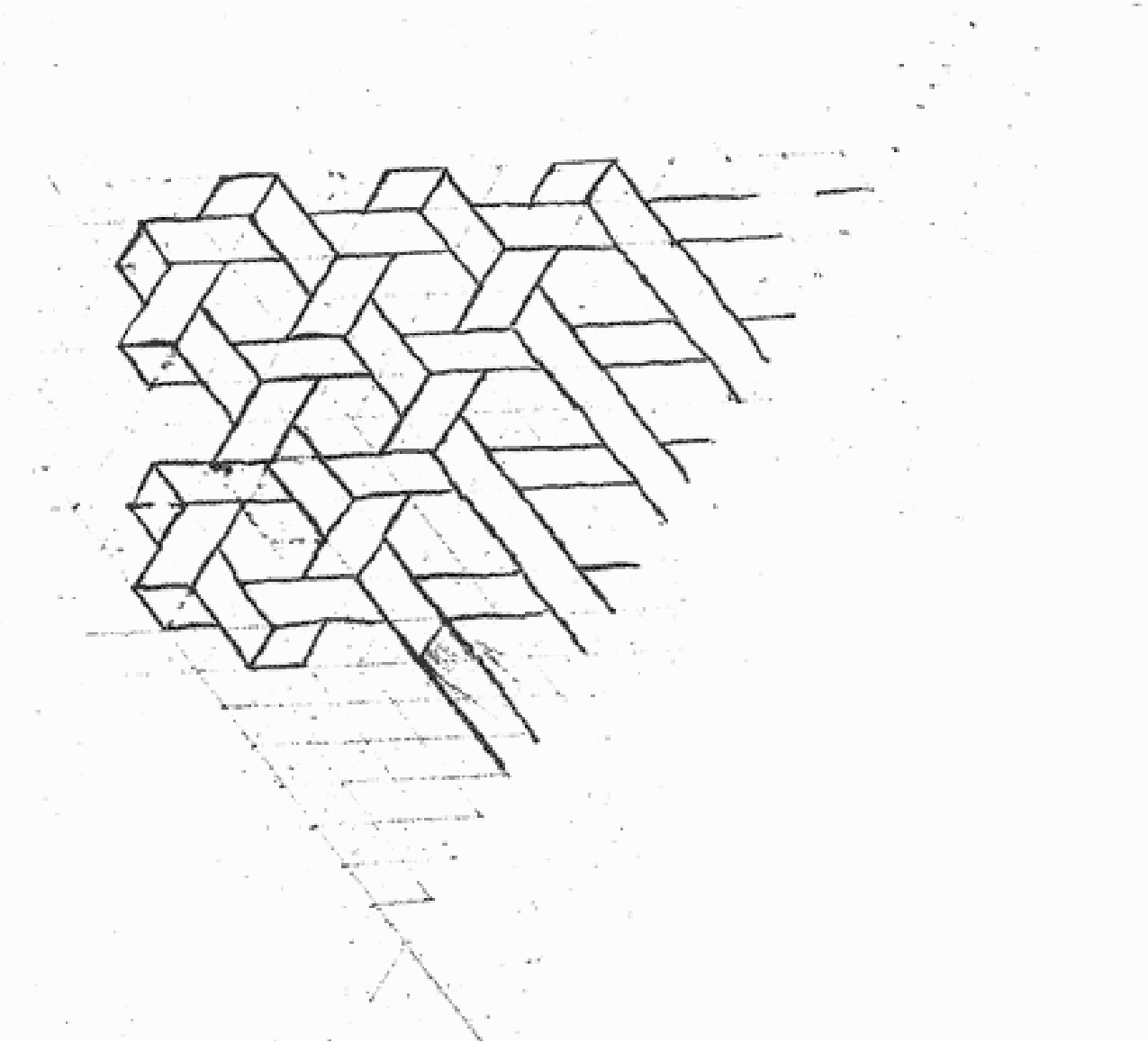}}
\caption{Framed tangle from Gauss' notebook \cite{Ga-1}}
\label{gauss12Arx}
\end{figure}
%\vspace*{4.4in}
%\centerline{\psfig{figure=gauss12Arx.eps,height=10.6cm}}
%\centerline{\psfig{figure=gauss12Ra.eps,height=10.6cm}}
%\centerline{\psfig{figure=gauss12.eps,height=10.6cm}}
%\centerline{Fig. 3.5; \ Framed tangle from Gauss' notebook \cite{Ga-1}}

In his note (Jan. 22 1833) Gauss introduces the linking number of two
knots\footnote{His method is analytical - the Gauss integral;
in modern language Gauss integral computes analytically the degree of
the map from a torus parameterizing a 2 component link to the unit 2-sphere.}.
Gauss' note presents the first deep incursion into knot theory;
it establishes that the following two links
are substantially different: \parbox{1.25cm}{\psfig{figure=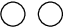}} and \parbox{0.8cm}{\psfig{figure=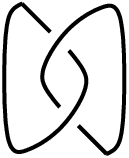,height=0.8cm}}. 
Gauss' analytical method has recently been revitalized by Witten's approach to knot theory \cite{Wit}.

James Clerk Maxwell (1831-1879), in his fundamental book of 1873
``A treatise on electricity \& magnetism" \cite{Max} writes
\footnote{It was only six years after
Gauss note was first published in his collected works in 1867.}
:
``It was the discovery by Gauss of
this very integral, expressing the work done on a magnetic pole while
describing a closed curve in presence of a closed electric current,
and indicating the geometrical connection between the two closed curves,
that led him to lament the small progress made in the Geometry
of Position since the time of Leibnitz, Euler and Vandermonde. We
have now, however, some progress to report chiefly due to Riemann,
Helmholtz and Listing."\footnote{Gauss wrote in 1833, in the same note in
which he introduced the linking number:
``On the geometry of position, which Leibniz initiated and to which only two
geometers,  Euler and Vandermonde,  have given a feeble glance,  we know and
possess, after a century and a half, very little more than nothing."}
Maxwell
goes on to describe two closed curves which cannot be separated
but for which the value of the  Gauss integral is equal to zero; Figure \ref{Maxwell}.

%\vspace*{1.5in}
\begin{figure}
\centering
\scalebox{1.45}{\includegraphics{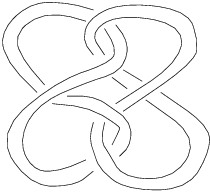}}
\caption{The link of Maxwell}
\label{Maxwell}
\end{figure}
%\centerline{\psfig{figure=Maxwell.eps}}\begin{center} Figure 3.6;\ The link of Maxwell. \end{center}

In 1876, O.~Boeddicker observed that, in a certain sense, the linking
number is the number of the crossing points of the second curve with
a surface bounded by the first curve \cite{Boe-1,Boe-2,Bog}.\
Hermann Karl Brunn\footnote{Born August 1, 1862 Rome (Italy),
died Sept. 20, 1939 (Munchen, Germany)\cite{Bl}.}
 \cite{Br} observed in 1892
that the linking number
of a two-component link, considered by Gauss, can be read from
a diagram of the link\footnote{It is also noted by Tait
in 1877 (\cite{Ta}, page 308).}.
If the link has components $K_1$ and $K_2$, we
consider any diagram of the link and count each point
at which $K_1$ crosses under $K_2$ as $+1$ for
\parbox{.9cm}{\psfig{figure=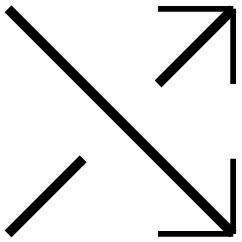,height=.5cm}}\
and $-1$ for\ \parbox{0.9cm}{\psfig{figure=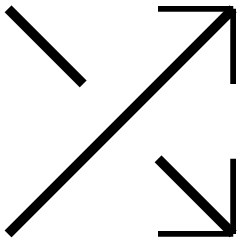,height=.5cm}}\ .
The sum of these, over all crossings
of $K_1$ under $K_2$, is the Gauss linking number.

   The year of 1847 was very important for the knot theory
(graph theory and topology as well).
On one hand,  Gustav Robert Kirchhoff (1824-1887)
published his fundamental paper on electrical circuits \cite{Kir}.
It has deep connections with knot theory,
however the relations were discovered only about
a hundred years later (e.g. the Kirchhoff complexity of a circuit corresponds
to the determinant of the knot or link determined by the circuit).
On the other hand, Johann Benedict Listing (1808-1882), a student of Gauss,
published his monograph (Vorstudien zur Topologie, \cite{Lis}).
A considerable part
of the monograph is devoted to knots. Even earlier, on April 1, 1836,
 Listing wrote a letter from Catania to "Herr Muller", his former school
teacher\footnote{Johann Heinrich M\"uller (1787-1844) was the mathematics
and astronomy master at {\it Musterschule} in Frankfurt which Listing entered in 1816
\cite{Bre1,Bre2}.}, with the heading
"Topology" [Its main contents were later incorporated in \cite{Lis}; see \cite{Bre2}. The entire text of the letter to M\"uller is published in \cite{Bre1}.
\ \\ \ \\

One can say much more on History of Knot Theory, some of this is described in my new book with PhD students \cite{PBIMW} where, in particular, the history of Knot tabulation is in detail described starting from Tait and Little and ending on B.Burton (who fully classified 19 crossings knots, and is working on extending his census of knots  up to 21 crossings). 
But I should not extend my first talk too much and in next talk I will concentrate on nonassociativity and Yang-Baxter homology and the talk will describe my personal rather convoluted  path with many digressions. 

The history of the Jones and HOMFLYPT polynomials will be discussed in detail elsewhere (starting from Jones discovery in May 1984, \cite{Jon1}, and our work with Pawe{\l} Traczyk in November and December of 1984; compare Subsection \ref{Entropic}).

\section{Attraction of nonassociative structures}\label{Attraction}
\subsection{Barbara Roszkowska question}
In 1985 or maybe in Spring of 1986 Basia\ \\
Roszkowska (PhD student of Anna Romanowska) told me that people from her group are interested in new objects in nonassociative algebra called quandles and a quandle is motivated by Knot Theory or more precisely Reidemeister moves \cite{Joy}. According to her I was saying that I am involved in generalizations of the Jones polynomial and that it is more interesting than quandles.
I do not remember our discussion so shame on me\footnote{But I should add ``Co si{\c e} odwlecze to nie uciecze" (``What is delayed will not run away" in Google translation). {\color{red} Added for e-print: See for example a recent paper with Ania Zamojska-Dzienio and my students on homology of Bol-Moufang quasigroups \cite{CCGOPZ}. }}. I remember however that their 
Polytechnic group organized a conference and that they may be interested in our (with Pawe{\l} Traczyk) work on generalizations of Jones polynomial, skein polynomial (later called HOMFLYPT polynomial). I wanted to go there, I was already at PKS (bus) stop, but it was so crowded that I finally gave up travel, and I got cold as it was very cold. Thus I lost there an early chance to cooperate with the Polytechnic group. Later I learned that one of the leaders of the group Jonathan Smith from Iowa State, was thinking about connecting his work with knot theory and even wrote a paper on the topic \cite{Smith}.
\subsection{Conway algebras}
In our work, with Pawe{\l} Traczyk, on generalizing Jones polynomial, November-December 1984, we formalized the 
fact that calculation should be independent on the order of resolving crossings. This led us to the concept of an abstract algebra which we called {\it Conway algebra}. We recall the definition in Subsection
\ref{Entropic}.

\begin{figure}
\centering
\scalebox{.35}{\includegraphics{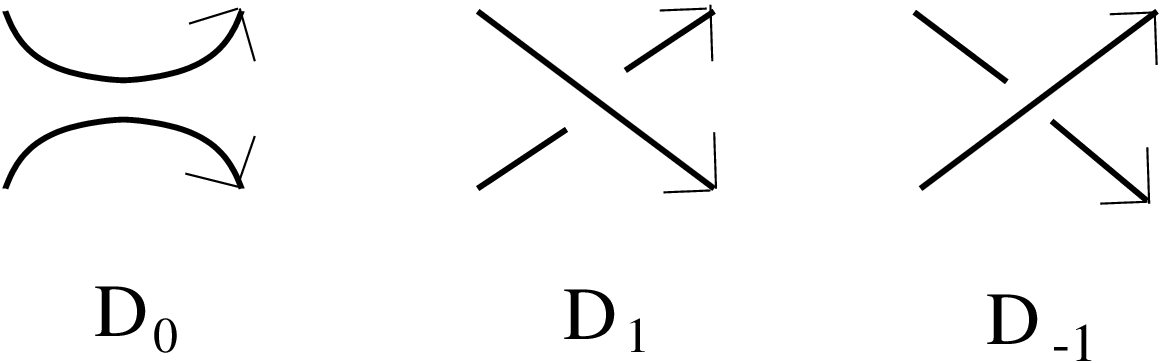}}
\caption{Oriented skein triple for the skein (HOMFLYPT) polynomial}
\label{skein-triple}
\end{figure}

\subsection{Link invariants and entropic condition}\label{Entropic}
In the papers with Pawel Traczyk \cite{PT1,PT2} we introduced a skein polynomial of links in $\mathbb R^3$ using linear skein relation:
$xD_1 +yD_{-1} =z D_0$ where $D_1,D_{-1} $ and $D_0$ are link diagrams identical 
outside some small disk and inside the disk they look like in the Figure \ref{skein-triple}. We were assuming that $x$ and $y$ are invertible but we also assumed  quickly (too quickly) that $z$ is also invertible. Thus we just put $z=1$. 
What is important here, we were musing whether a linear equation could be replaced 
by some magma, or more generally abstract algebra $\mathcal A$ with universe $A$ and two  binary operations, $*$ and $|$  such that:
$D_{-1}=D_1*D_0$ and $D_1=D_{-1}|D_0$.
Then an obvious condition for a linear relation that calculation does not depend on the order of crossing we resolve (the condition being $xy=yx$) is replaced by an interesting, not obvious condition $(a*b)*(c*d)=(a*c)*(b*d)$.\footnote{The conditions $(a|b)*(c|d)=(a*c)|(b*d)$ and $(a|b)|(c|d)=(a|c)|(b|d)$ follow automatically in Conway algebra.} We did not know then, in Fall of 1984, that this condition was studied before: first in 1929 by Burstin\footnote{Celestyn Burstin (1888-1938)
was born in Tarnopol, where he obtained ``Matura" in 1907, he moved to Vienna where in 1911 he completed university.
In 1929 he moved to Minsk where he was a member of the Belarusian National Academy of Sciences,
and a Director of the Institute of Mathematics of the Academy. In December 1937,
he was arrested on suspicion of activity as a
spy for Poland and Austria. He died in October 1938, when interrogated in a prison in Minsk (``Minskaja Tjurma");
he was rehabilitated March 2, 1956 \cite{Bur-1,Bur-M,Mal,Mio}.} and Mayer\footnote{Walter Mayer (1887--1948) is well known for Mayer-Vietoris sequence and for being
assistant to A.~Einstein at Institute for Advanced Study, Princeton \cite{Isa}. In Princeton he started investigation of generalization of chain complexes where $\partial^k=0$ while in the classical theory $k=2$, \cite{Mayer}.}  \cite{BuMa}. The name entropic magma was coined by Etherington in 1949, \cite{Eth}.
In \cite{PT1,PT2} we introduce the notion of Conway algebras so let us remind their definition below.

\begin{definition}\label{ConwayAlgebra} We say that an abstract algebra 
$\mathcal A= \{A; a_1,\,a_2,\ldots,\, |,\, * \}$ composed of a set $A$, two binary operations $|$ and $*$ and a countable number of $0$-argument operations $a_i$, $i\geq 1$, is a Conway algebra if the following conditions are satisfied:
$$
\left.
\begin{array}{llll}
\mbox{C1} & a_n| a_{n+1} &=& a_n\\
\mbox{C2} & a_n \ast a_{n+1}& =& a_n\\
\end{array}
\right\} 
\mbox{ initial values properties}
$$

$$
\left.
\begin{array}{llll}
\mbox{C3} & (a| b)| (c| d)&=&(a|
c)|(b| d) \\
\mbox{C4} &(a| b)\ast(c| d)&=&(a\ast
c)|(b\ast d) \\
\mbox{C5} &(a\ast b)\ast(c\ast d)&=&(a\ast c)\ast(b\ast d) \\
\end{array}
\right\} 
\mbox{entropy properties}
$$

$$
\left.
\begin{array}{llll}
\mbox{C6} &(a| b)\ast b &=& a\mbox{\ \ \ \ \ \ \ } \\
\mbox{C7} &(a\ast b)| b &=& a \\
\end{array}  
\right\} 
\mbox{ inversion properties.}
$$
\end{definition}
We proved with Pawe{\l} Traczyk that any Conway algebra leads to link invariant. From the nonassociative point of view we showed that entropy axiom corresponds to the fact that computation of the polynomial does not depend on the order of resolving different crossings. We will demonstrate this fact in Subsection~\ref{skein-entropic} for any local skein relation that is performed far apart from one another (i.e.,``no action at a distance").

\subsection{Local skein relations and general entropic relation}\label{skein-entropic}

It was important historically to see that computing Kauffman bracket or HOMFLYPT 
polynomial of links in $\mathbb R^3$ does not depend on the order of crossings we resolve. As mentioned before, this let the authors of \cite{PT1} to consider abstract algebra which they called Conway algebra and which satisfies entropic relation $(a*b)*(c*d)=(a*c)*(b*d)$.
In general 3-manifold the property should be somehow reformulated.
The setting is as follows: \\
consider two disjoint oriented balls $B_1^3$ and $B_2^3$ in an oriented 3-dimensional manifold $M$ and fix the part of a link outside  $B_1^3$ and $B_2^3$.
We now consider skein relations taking place only inside $B_1^3$  and $B_2^3$.
We ask whether a calculation done in $B_1^3$ before calculation in $B_2^3$ gives the same result as calculation done first in $B_2^2$ and then in $B_1^3$. 

Let us try to write it in much general terms, again stressing the fact that changing the order of disks $B_1^3$ and $B_2^3$ produces the same result if we work with 
commutative ring $R$. Let us fix notation: $D^{B_1,B_2}_{T_1,T_2}$ denote a link 
which is fixed outside $B_1^3$ and $B_2^3$ and in $B_1^3$ and $B_2^3$ has a tangle 
$T_1$ and $T_2$ respectively. Assume that in $B_1^3$ we have a skein relation 
$\sum_{i=0}^{n_1} b_i^{(1)}T_i^{(1)}=0$ and in $B_2^3$ we have a skein relation 
$\sum_{i=0}^{n_2} b_i^{(2)}T_i^{(2)}=0$. Assume also, for simplicity that 
$b_{n_1}^{(1)}$ and $b_{n_2}^{(2)}$ are invertible in the ring $R$. Then we can 
compute $D^{B_1,B_2}_{T_{n_1}^{(1)},T_{n_2}^{(2)}}$ starting from the relation in $B_1^3$ or from relation in $B_1^3$ and for commutative $R$ the result is independent on the order of chosen balls $B_i^3$.
\begin{proof}
Let us perform calculation starting from the relation in the disk $B_1^3$. We have:
\begin{eqnarray*}
D^{B_1,B_2}_{T_{n_1}^{(1)},T_{n_2}^{(2)}} &=& (-b_{n_1}^{(1)})^{-1}\sum_{i=0}^{n_1-1} b_i^{(1)}D^{B_1,B_2}_{T_{i}^{(1)},T_{n_2}^{(2)}}\\
&=& (-b_{n_1}^{(1)})^{-1}\sum_{i=0}^{n_1-1} b_i^{(1)}\bigg((-b_{n_2}^{(2)})^{-1}\sum_{j=0}^{n_2-1} b_j^{(2)}D^{B_1,B_2}_{T_{i}^{(1)},T_{j}^{(2)}}\bigg)\\
&=& (b_{n_1}^{(1)})^{-1}\sum_{0\leq i < n_1, 0\leq j < n_2}b_i^{(1)}(b_{n_2}^{(2)})^{-1}b_j^{(2)}D^{B_1,B_2}_{T_{i}^{(1)},T_{j}^{(2)}}.
\end{eqnarray*}
If we use the skein relation in the disk $B_2^3$ first, we get analogously:
\begin{equation*}
D^{B_1,B_2}_{T_{n_1}^{(1)},T_{n_2}^{(2)}}= (b_{n_2}^{(2)})^{-1}\sum_{0\leq i < n_1, 0\leq j < n_2}b_j^{(2)}(b_{n_1}^{(1)})^{-1}b_i^{(1)}D^{B_1,B_2}_{T_{i}^{(1)},T_{j}^{(2)}}.  
\end{equation*}
In commutative ring $R$ we have $b_i^{(1)}b_j^{(2)}=b_j^{(2)}b_i^{(1)}$ for all $i$ and $j$ so both expressions are equal.
\end{proof}
%\subsection{General entropy condition}\label{Entropic}
We are ready now to bring to our consideration the general entropy condition. 
We follow here mostly \cite{SmRo,RoSm}.\\
Consider a set $X$ with two operations, one $n$-ary, $*_1: X^n \to X$ and the second $m$-ary, $*_2: X^m \to X$.
We say that $(X;*_1,*_2)$ is a generalized entropic algebra (that is satisfies the entropic condition) if
for $mn$ elements of $X$, say $D_{i,j}$ where $1\leq i \leq m$ and $1\leq j \leq n$ we have entropic condition:
\begin{equation*}
*_2\bigg(*_1(r_1),*_1(r_2),\ldots,*_1(r_m)\bigg) = *_1\bigg(*_2(c_1),*_2(c_2),\ldots,*_2(c_n)\bigg),    
\end{equation*}
where 
\begin{equation*}
r_i=(D_{i,1},D_{i,2},\ldots,D_{i,n})\,\,\text{and}\,\, c_j=(D_{1,j},D_{2,j},\ldots,D_{m,j})
\end{equation*}
are the rows and columns, respectively of the matrix $\{D_{i,j}\}$ described below.

\[ \{D_{i,j}\}=
\left[
\begin{array}{cccccc}
D_{1,1} & D_{1,2} & D_{1,3} &  \cdots & D_{1,n-1} & D_{1,n} \\
D_{2,1} & D_{2,2} & D_{2,3} &  \cdots & D_{2,n-1} & D_{2,n} \\
D_{3,1} & D_{3,2} & D_{3,3} &  \cdots & D_{3,n-1} & D_{3,n} \\
\cdots  & \cdots   &\cdots  &\cdots   &\cdots     & \cdots  \\
D_{m-1,1} & D_{m-1,2} & D_{m-1,3} & \cdots & D_{m-1,n-1} & D_{m-1,n} \\
D_{m,1}   & D_{m,2}   & D_{m,3}   & \cdots & D_{m,n-1}   & D_{m,n}
\end{array}
\right],
\]
We can reformulate our condition by saying that $*_2$ is a homomorphism of the magma $(X,*_1)$ and $*_1$ is a
homomorphism of the magma $(X,*_2)$.
The case of $m=2=n$ is the classical case and the entropic condition can be written as 
\begin{equation*}
(a*_1b)*_2(c*_1d)=(a*_2c)*_1(b*_2d).
\end{equation*}
The name ``entropic" means ``inner turning" refering to swamping $b$ and $c$ (see \cite{RoSm,SmRo}.)
Entropic magma was considered for the first time in \cite{BuMa} in 1929\footnote{The work of A.K. Suschkewitsch should be also mentioned \cite{Sus,Hol}.} and the name entropic magma was coined in \cite{Eth} in 1949.
\subsection{\texorpdfstring{$t_k$\,- moves and link invariants}{t\unichar{"005E}2\unichar{"00D7}S\unichar{"005E}1}}\label{3.5}
At the beginning of 1986, I thought a little about the multivariable Alexander polynomial (useful a year later when I constructed (or discovered) skein modules). I was motivated by the original Conway paper \cite{Con}. This consideration led me, in the Spring of 1986, to analyze what I called $t_k$ and $\bar t_{2k}$ moves (for unoriented links I adopted H. Morton suggestion and called our moved, $k$-moves, see Figures \ref{3-move-trefoil-feightkn} and \ref{Quandle-n-move}). This led to the paper \cite{Prz1}. I liked the paper a lot and was coming back to $k$-moves often. One example was analysis of $k$-moves for a general skein relation (and skein module $\mathcal S_{m+1,\infty}(M))$, see recent paper \cite{BCGIKMMPW} devoted to cubic skein modules, $\mathcal S_{4,\infty}(M)$, (in preparation).
An analysis of $k$-moves on Khovanov homology \cite{Kho} is still waiting.\footnote{See however the work by Wojtek Politarczyk, Maciej Borodzik, and Marithania Silvero  \cite{BPS}.} Figure \ref{Quandle-n-move} illustrate the effect on $n$-move on a distributive structure (quandle).

Naturally, I did not have time to provide details related to $n$-moves during my second talk in June 2023 however, in this extended version of my talk, I would like, at least, to show one basic and important calculation in the following Subsection.

\subsection{Two points of view on \texorpdfstring{$n$\,-moves formula}{n\unichar{"005E}2\unichar{"00D7}S\unichar{"005E}1}}\label{TwoPoint}
The calculation I have in mind concern affine recursive relation and $n$ moves on link diagrams. Here is the story:\\
In linear algebra or combinatorics we often meet recursive relations of type
\begin{equation*}
Q_m = b_{m-1}Q_{m-1}+b_{m-2}Q_{m-2}+\ldots+b_0Q_0 +c_mQ_c
\end{equation*}
where $Q_i$, $Q_c$ are some additive objects, say elements of a $k$-module over commutative ring $k$ and 
$b_i,c_n$ are the ring elements. 
It is well known, from XIX century how to solve this (even in Maxwell work I did see related considerations). Similar recursion often happen in knot theory with additional twist that we can use sometimes topological ``tricks". I will describe here one such situation which I first noticed in \cite{Prz1} and we use it also very recently in 
our analysis (Mathathon 8 \cite{BCGIKMMPW} )of $n$ moves on cubic skein modules. The longish calculation would follow (of course I skipped it in my talk).

%{\color{red} I have to decide which paper it will go to.}\\
As we often do in mathematics, at least from the times of Euler, we count (compute) some objects in two different ways and often get interesting identities. We use the same principle to see how $n$th move on a link diagram is computed using  $m$th degree skein relation. 
In fact to have good understanding of our calculation let us perform it 
in the case of a quartic relation first (quartic skein module)
Later it would be natural to extend our formulas for arbitrary $m$. In fact this ``Eulerian" trick was first used for quadratic relation (Kauffman and Dubrovnik skein relations, see \cite{Prz1,DIP}). \\
Thus let us consider the quartic relation; for simplicity in the basic quartic relation
\begin{equation*}
b_4D_4+b_3D_3+b_2D_2+b_1D_1+b_0D_0+ b_\infty D_{\infty}=0
\end{equation*}
(see Figure \ref{quartic-relation}), we put $b_4=-1$ and
get more manageable recursive relation $$D_4=b_3D_3+b_2D_2+b_1D_1+b_0D_0+ b_\infty D_{\infty}.$$
Considering the denominator of this relation we get the formula for $b_\infty t$ where $t$ represent the trivial component ($tD = D\sqcup T_1$); see Figure \ref{quartic-relation-den2}.
$$a^{-4}=a^{-3}b_3+a^{-2}b_2+ a^{-1}b_1+b_0+b_\infty t$$
so
\begin{equation}\label{t-formula}
b_\infty t= a^{-4}-a^{-3}b_3 - a^{-2}b_2- a^{-1}b_1-b_0
\end{equation}
and for invertible $b_\infty$ we can compute $t$.\\
It is very educational to find few terms of recursive resolution 
of $D_n$ using quartic relation ($\stackrel{(k)}{=} $ means that we performed  recursive relation $k$ times).
\begin{eqnarray*}
D_n&\stackrel{(1)}{=}& b_3D_{n-1}+b_2D_{n-2}+b_1D_{n-3}+b_0D_{n-4}+ a^{4-n}b_{\infty}D_{\infty} \stackrel{(2)}{=}\\
&&b_3(b_3D_{n-2}+b_2D_{n-3}+b_1D_{n-4}+b_0D_{n-5}+ a^{5-n}b_{\infty}D_{\infty})+\\
&&b_2D_{n-2}+b_1D_{n-3}+b_0D_{n-4}+ a^{4-n}b_{\infty}D_{\infty}=\\
&&(b_3^2+b_2)D_{n-2}+(b_2b_3+b_1)D_{n-3}+(b_1b_3+b_0)D_{n-4}+b_0b_3D_{n-5}+\\
&&(b_3a^{4-n}+a^{5-n})b_{\infty}D_{\infty} \stackrel{(3)}{=}\\
&&(b_3^2+b_2)(b_3D_{n-3}+b_2D_{n-4}+b_1D_{n-5}+b_0D_{n-6} + a^{6-n}b_{\infty}D_{\infty})+\\
&&(b_2b_3+b_1)D_{n-3}+(b_1b_3+b_0)D_{n-4} + b_0b_3D_{n-5}+
(b_3a^{5-n}+a^{4-n})b_{\infty}D_{\infty}=\\
&&(b_3^3+2b_2b_3+b_1)D_{n-3}+(b_2b_3^2+b_2^2+b_1b_3+b_0)D_{n-4}+(b_1(b_3^2+b_2)+b_0b_3)D_{n-5}+\\
&&b_0(b_3^2+b_2)D_{n-6}+ ((b_3^2+b_2)a^{6-n}+b_3a^{5-n}+a^{4-n})b_{\infty}D_{\infty}\stackrel{(4)}{=}
\end{eqnarray*}
\begin{eqnarray*}
&&(b_3^3+2b_2b_3+b_1)(b_3D_{n-4}+b_2D_{n-5}+b_1D_{n-6}+b_0D_{n-7}+a^{7-n}b_{\infty}D_{\infty})+\\
&&(b_2b_3^2+b_2^2+b_1b_3+b_0)D_{n-4}+(b_1(b_3^2+b_2)+b_0b_3)D_{n-5}+\\
&&b_0(b_3^2+b_2)D_{n-6}+ ((b_3^2+b_2)a^{6-n}+b_3a^{5-n}+a^{4-n})b_{\infty}D_{\infty}=\\
&&(b_3(b_3^3+2b_2b_3+b_1)+b_2(b_3^2+b_2)+b_1b_3+b_0)D_{n-4}+\\
&&(b_2(b_3^3+2b_2b_3+b_1)+b_1(b_3^2+b_2)+b_0b_3)D_{n-5}+\\
&&(b_1(b_3^3+2b_2b_3+b_1)+b_0(b_3^2+b_2))D_{n-6}+b_0(b_3^3+2b_2b_3+b_1)D_{n-7}+\\
&&\bigg((b_3^3+2b_2b_3+b_1)a^{7-n}+ (b_3^2+b_2)a^{6-n}+b_3a^{5-n}+a^{4-n}\bigg)b_{\infty}D_{\infty}.
\end{eqnarray*}
From this we can see a pattern, easily proven by induction, using the sequence $P_n$ defined by 
$P_{-3} = P_{-2} = P_{-1} = 0,\,P_0 = 1$ and a recursive relation
\begin{equation*}
P_n=b_3P_{n-1}+b_2P_{n-2}+b_1P_{n-3}+b_0P_{n-4}
\end{equation*}
(to list more terms, we have: $P_1=b_3$, $P_2=b_3^2+b_2$, $P_3=b_3^3+2b_2b_3+b_1$.) 
We can continue our calculation of $D_n$ using our polynomials $P_n$ to get:
\begin{eqnarray}
\label{4nk-formula}
D_n &\stackrel{(4)}{=}& P_4D_{n-4} + (b_2P_3+b_1P_2+b_0P_1)D_{n-5}+(b_1P_3+b_0P_2)D_{n-6}+b_0P_3D_{n-1}+\\ \notag
&&\bigg(\sum_{i=0}^3 a^{4-n+i}P_i\bigg)b_\infty D_\infty =\ldots \stackrel{(k)}{=}\\ \notag
&&P_kD_{n-k} +(b_2P_{k-1}+b_1P_{k-2}+b_0P_{k-3})D_{n-k-1}+\\ \notag
&&(b_1P_{k-1}+b_0P_{k-2})D_{n-k-2} + b_0P_{k-1}D_{n-k-3}+\bigg(\sum_{i=0}^{k-1} a^{4-n+i}P_i\bigg)b_\infty D_\infty.
 \end{eqnarray}
We introduce notation 
\begin{equation*}
U^{(4)}_{n,k}=\sum_{i=0}^{k-1} a^{4-n+i}P_i;
\end{equation*}
upper index $4$  refers to the fact that we deal with quartic skein relation. For example, we use often the formula for $D_n$ in the case of $k=n-3$ in that case we get:
\begin{eqnarray}
\label{DnD1D0}
D_n&\stackrel{(n=k-3)}{=}&
P_{n-3}D_3+ (b_2P_{n-4}+b_1P_{n-5}+b_0P_{n-6})D_2 +  (b_1P_{n-4}+b_0P_{n-5})D_1 \\ \notag
&+&b_0P_{n-4}D_{0} +U^{(4)}_{n,n-3}b_{\infty}D_{\infty}.
\end{eqnarray}
We will show now how to find quickly the closed formula for $U^{(4)}_{n,k}$. It is the standard method, pioneered by Euler, to look at a  problem from two different perspectives and to get some nice identity comparing results. In fact this ``Eulerian  trick" was used for quadratic relation (Kauffman and Dubrovnik skein relations, see \cite{Prz1,DIP}). We have already formula for $U^{(4)}_{n,k}$ as a sum. Alternative approach is to consider denominators of diagrams used in Formula (\ref{4nk-formula}); see Figure \ref{quartic-relation-den2}:
\begin{eqnarray}
\label{Dn4-denom}
a^{-n}&=& a^{k-n}P_{k} + a^{k-n+1}(b_2P_{k-1}+b_1P_{k-2}+b_0P_{k-3})\\ \notag
&+&a^{k-n+2}(b_1P_{k-1}+b_0P_{k-2})+a^{k-n+3}b_0P_{k-1}+U^{(4)}_{n,k}b_{\infty} t 
\end{eqnarray}

Assuming that $b_\infty t$ is invertible and substituting  Formula \ref{t-formula} into Formula \ref{Dn4-denom} we obtain close formula for $U^{(4)}_{n,k}$:
\begin{eqnarray}
\label{closedFormula4} 
U^{(4)}_{n,k}&=&\sum_{i=0}^{k-1} a^{4-n+i}P_i = \frac{a^{k-n}P_{k} + a^{k-n+1}(b_2P_{k-1}+b_1P_{k-2}+b_0P_{k-3})}{a^{-3}b_3 + a^{-2}b_2 + a^{-1}b_1+b_0-a^{-4}}\\ \notag
&+&\frac{a^{k-n+2}(b_1P_{k-1}+b_0P_{k-2})+a^{k-n+3}b_0P_{k-1}-a^{-n}}{a^{-3}b_3 + a^{-2}b_2 + a^{-1}b_1+b_0-a^{-4}}
\end{eqnarray}

\begin{figure}[hp]
\centering
\scalebox{.19}{\includegraphics{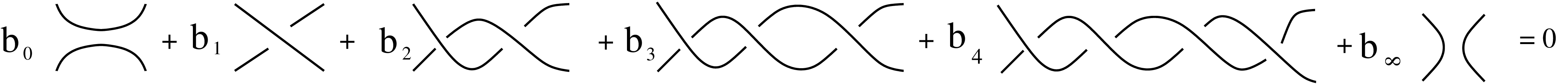}}
\caption{Quartic relation}
\label{quartic-relation}
\end{figure}

\begin{figure}[hp]
\centering
\scalebox{.19}{\includegraphics{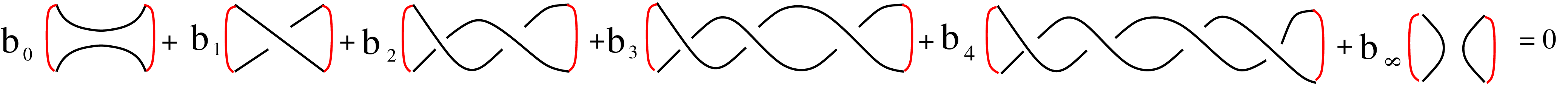}}
\caption{Denominator of the quartic relation; 
$\big(b_0+a^{-1}b_1+a^{-2}b_2+ a^{-3}b_3+a^{-4}b_4+ b_\infty t\big)t=0 $. Compare Equation \ref{t-formula} with $b_4=-1$} 
\label{quartic-relation-den2}
\end{figure}

Having already Formula (\ref{closedFormula4}) we can prove it quickly by induction without referring to ``topological trick". 
\begin{proof} We use induction on $k$. For $k=1$ we see directly that both sides are equal to $a^{4-n}$. Now we consider $k>1$ and assume (inductive assumption) that Formula (\ref{closedFormula4}) holds for $k-1$. We calculate:
\begin{eqnarray*}
&&(a^{-3}b_3 + a^{-2}b_2 + a^{-1}b_1+b_0-a^{-4})(\sum_{i=0}^{k-1} a^{4-n+i}P_i)=\\
&&(a^{-3}b_3 + a^{-2}b_2 + a^{-1}b_1+b_0-a^{-4})((\sum_{i=0}^{k-2} a^{4-n+i}P_i) +a^{3-n+k}P_{k-1})\stackrel{Induction}{=}\\
&&a^{k-n-1}P_{k-1}+ a^{k-n}(b_2P_{k-2}+b_1P_{k-3}+b_0P_{k-4})+a^{k-n+1}(b_1P_{k-2}+b_0P_{k-3})+\\
&& a^{k-n+2}b_0P_{k-2}-a^{-n}
+(a^{-3}b_3 + a^{-2}b_2 + a^{-1}b_1+b_0-a^{-4})a^{3-n+k}P_{k-1}= \\
&&a^{k-n}(b_3P_{k-1}+ b_2P_{k-2}+b_1P_{k-3}+b_0P_{k-4})+\\ 
&& a^{k-n+1}(b_2P_{k-1}+ b_1P_{k-2}+b_0P_{k-3})
+ a^{k-n+2}(b_1P_{k-1}+ b_0P_{k-2})+\\
&& a^{k-n+3}b_0P_{k-1} + a^{k-n-1}P_{k-1}-a^{k-n-1}P_{k-1}-a^{-n}= \\
&& a^{k-n}P_k + a^{k-n+1}(b_2P_{k-1}+ b_1P_{k-2}
+b_0P_{k-3})+ a^{k-n+2}(b_1P_{k-1}+ b_0P_{k-2})+\\
&& a^{k-n+3}b_0P_{k-1} -a^{-n}
\end{eqnarray*}
as needed.
\end{proof}

\subsection{\texorpdfstring{$n$-move in the $(m+1)$ ($m$th degree) skein module}{n\unichar{"005E}2\unichar{"00D7}m\unichar{"005E}1}}\label{nmovemdegree}\ \\
The effect on an $n$-move on $m$th degree recursive skein relation is analogous (direct generalization) but without having small cases (like $m=4$) computed before it would be difficult to perform.\\
Let us consider the $m$ degree skein relation:
\begin{equation*}
b_mD_m+b_{m-1}D_{m-1}+ \ldots + b_1D_1 +b_0D_0+ b_\infty D_\infty =0 \mbox{ and } D^{(1)}=aD.    
\end{equation*}
From this we get, assuming $b_m = -1$:
\begin{eqnarray*}
D_n&=&b_{m-1}D_{n-1} + b_{m-2}D_{n-2}+ b_{m-3}D_{n-3}+\ldots + b_{1} D_{n-m+1}+ b_0D_{n-m}+ a^{m-n}b_{\infty}D_{\infty}\\
&=& b_{m-1}(b_{m-1}D_{n-2} + b_{m-2}D_{n-3}+ \ldots + b_{1} D_{n-m}+ b_0D_{n-m-1}+ a^{m-n+1}b_{\infty}D_{\infty})\\
&+&b_{m-2}D_{n-2}+  b_{m-3}D_{n-3}+\ldots + b_1 D_{n-m+1}+ b_0D_{n-m}+ a^{m-n}b_{\infty}D_{\infty}\\
&=&(b_{m-1}^2+b_{m-2})D_{n-2}+ (b_{m-2}b_{m-1}+ b_{m-3})D_{n-3}+ \ldots +(b_0b_{m-1}+b_0)D_1\\
&+& b_0b_{m-1}D_0 + (b_{m-1}a^{m-n+1}+ a^{m-n})b_{\infty}D_{\infty}\\
&=&(b_{m-1}^2+b_{m-2})(b_{m-1}D_{n-3}+b_{m-2}D_{n-4}+b_{m-3}D_{n-5}+a^{m-n+2}b_\infty D_\infty)\\
&+& (b_{m-2}b_{m-1}+ b_{m-3})D_{n-3}+ \ldots +(b_0b_{m-1}+b_0)D_1 + b_0b_{m-1}D_0\\
&+&(b_{m-1}a^{m-n+1}+ a^{m-n})b_{\infty}D_{\infty}\\
&=&P_{3,0}^{(m)}D_{n-3}+ P_{3,1}^{(m)} D_{n-4}+ P_{3,2}^{(m)}D_{n-5}+\ldots+ P_{3,m-3}^{(m)}D_{n-m}+U_{n,3}^m. 
\end{eqnarray*}
We can introduce notation which will allow us to write formula in compact way: First $P_{n,s}^{m}$ defined first for $s=0$:
\begin{equation*}
P^{(m)}_{j,0}=0\,\,\text{for}\,\,-m <j <0\,\,\,P^{(m)}_{0,0}=1
\end{equation*}
and the recursive relation is 
$$P_{n,0}^{(m)} = b_{m-1}P_{n-1,0}^{(m)}+b_{m-2}P_{n-2,0}^{(m)}+...+b_{1}P_{n-m+1,0}^{(m)}+b_0P_{n-m,0}^{(m)}.$$
Thus we have: 
\begin{equation*}
P_{-m,0}^{(m)}=\frac{1}{b_0},\,P_{1,0}^{(m)}=b_{m-1},\, P_{2,0}^{(m)}=b_{m-1}^2+b_{m-2},\,\text{etc.}  
\end{equation*}
We define $P^{(m)}_{k,s}$ by the formula:
$$P^{(m)}_{k,s}= \sum_{j=0}^{m-1} b_{m-1-j-s}P^{(m)}_{k-1-j}.$$
Notice that for $s=0$ it is our recursive definition of $P^{(m)}_{k,0}$ 
(which we will be denoted shortly $P^{(m)}_k$ or $P_k$ if $m$ is fixed for a given part of the paper).
The formula can be easily checked by induction. The special cases, very useful to produce the general formula, were studied for $m=2,3,4$, see Formulas (\ref{4nk-formula}) and (\ref{closedFormula4}) for the quartic case ($m=4$).\\
We can also find quickly a closed formula for $U^{(m)}_{n,k}$ by considering the denominator of the formula for $D_n$.
\begin{eqnarray}
\label{Identity-m}
U^{(m)}_{n,k} &=& \sum_{i=0}^{k-1}a^{m-n+i}P_i = \frac{a^{-n}-P_{k,0}^{(m)} a^{k-n} - P_{k,1}^{(m)} a^{k-n+1} - \ldots - P_{k,m-1}^{(m)}a^{k-n+m-1}}{a^{-m}- a^{-m+1}b_{m-1} -a^{-m+2}b_{m-2} - \ldots -a^{-1}b_1-b_0}
\end{eqnarray}
For small $m$, say $m=4$ i.e. quartic  relation, we get a short closed formula, see Formula  (\ref{closedFormula4}). 
%\ref{4nk-formula}

\begin{proof}
We compare it with the denominator of the basic $m$th degree skein relation:
\begin{equation}\label{comp-trivial-m}
a^{-m}=a^{-m+1}b_{m-1} + a^{-m+2}b_{m-2}+ ... +a^{-1}b_1 +b_0 +b_\infty t
\end{equation}
Equivalently:
$$b_\infty t= a^{-m}- a^{-m+1}b_{m-1} - a^{-m+2}b_{m-2} -...-a^{-1}b_1-b_0= a^{-m}- \sum_{i=0}^{m-1} a^{-i}b_i. $$
From which, assuming $b_\infty$ is invertible, we can compute $t$.
\begin{eqnarray}
\label{MainEquation}
D_n &=& P_{k,0}^{(m)} D_{n-k} + P_{k,1}^{(m)} D_{n-k-1} + P_{k,2}^{(m)} D_{n-k-2} +\ldots \\ \notag
&+& P_{k,m-1}^{(m)} D_{n-k-m+1} + U^{(m)}_{n,k}b_\infty D_\infty = \sum_{i=0}^{m-1} P_{k,i}^{(m)} D_{n-k-i} + U^{(m)}_{n,k}b_\infty D_\infty.
\end{eqnarray}

The denominator of Equation (\ref{MainEquation}) yields the closed formula for $U^{(m)}_{n,k}$, as follows:
\begin{eqnarray*}
a^{-n} &=& P_{k,0}^{(m)} a^{k-n} + P_{k,1}^{(m)} a^{k-n+1} + P_{k,2}^{(m)} a^{k-n+2} + \ldots + P_{k,m-1}^{(m)}a^{k-n+m-1}  + U^{(m)}_{n,k}b_\infty t\\
&=& \sum_{i=0}^{m-1} P_{k,i}^{(m)}a^{k-n+i}  + U^{(m)}_{n,k}b_\infty t.
\end{eqnarray*}
%\begin{equation}\label{MainEquationDen}
%a^{-n}= P_{k,0}^{(m)} a^{k-n} + P_{k,1}^{(m)} a^{k-n+1} + P_{k,2}^{(m)} a^{k-n+2} + ... + P_{k,m-1}^{(m)}a^{k-n+m-1}  + U^{(m)}_{n,k}b_\infty t =
%\end{equation}
Therefore,
\begin{eqnarray*}
U^{(m)}_{n,k}b_\infty t &=& a^{-n}-P_{k,0}^{(m)} a^{k-n} - P_{k,1}^{(m)} a^{k-n+1} - \ldots - P_{k,m-1}^{(m)}a^{k-n+m-1}\\
&=&a^{-n}- \sum_{i=0}^{m-1} P_{k,i}^{(m)}a^{k-n+i}.
\end{eqnarray*}
Substituting in this Equation (\ref{comp-trivial-m}) we get
\begin{equation*}
U^{(m)}_{n,k}(a^{-m}- \sum_{i=0}^{m-1} a^{-i}b_i)=a^{-n}- \sum_{i=0}^{m-1} P_{k,i}^{(m)}a^{k-n+i},    
\end{equation*} 
which for $b_\infty t$ invertible is equivalent to Formula (\ref{MainEquation}).
\end{proof}
For a fixed $m$ one can find now the closed formula for $U^{(m)}_{n,k}$ (in general, Formula  (\ref{Identity-m}), and for $m=4$ Formula (\ref{closedFormula4}). As in the case of $m=4$, we can prove directly, by induction on $k$, the formula without using ``topological trick".

%\newpage
\subsection{Mietek D{\c a}bkowski and Burnside groups of links}\label{Mietek}
In the Summer of 1999, Mieczys{\l}aw (Mietek) D{\c a}bkowski came to study with me at GWU. He was recommended by my Gdansk friend Witold (Witek) Rosicki. He tried many open problems\footnote{One of them was to compute the Kauffman bracket skein module of the 3-manifold $F_{0,3}\times S^1$, where $F_{g,d}$ denotes the compact surface of genus $g$ ans $d$ boundary components. After several attempts Mietek  with Maciej (Maciek) Mroczkowski solved the problem, showing in particular that the module is free \cite{DM}. Another problem was to compute Kauffman bracket skein module of lattice crossings. We solved this partially later, see e.g. \cite{DLP,DaPr4}.} and the big success came with relation to Montesinos-Nakanishi 3-move conjecture. Y. Nakanishi  formulated  the conjecture in 1981. 
J. Montesinos
analyzed $3$-moves before, in connection with $3$-fold dihedral branch
coverings, and asked related but different question.
%\newpage
\begin{conjecture} Every link can be reduced to a trivial link by 3-moves (see Figure \ref{3-move-trefoil-feightkn} for 3-move and  reduction of the trefoil knot and the figure eight knot).
\end{conjecture}
The conjecture was proved in many special cases (e.g. for links up to $12$ crossings in \cite{Che}) but it was an
open problem for over 20 years. In 2002 it was showed by M.K.~D{\c a}bkowski and the author
that the conjecture does not hold. The smallest counterexample we found, suggested first by Q.~Chen,
has 20 crossings, see Figure \ref{Chen5braidTrieste}, \cite{DaPr1}.\footnote{One year after these talks, we proved that in fact the Montesinos-Nakaniahi 3-move conjecture holds for links up to 19 crossings. With respect to 20 crossings links we showed that there are six pairwise non isotopic links (all of them 5-braids of 20 crossings including the Chen link and its mirror image) which are not 3-move equivalent to any trivial link. All of them are 3-move equivalent to the Chen link (\cite{BBGIMMP}, in preparation).}

\begin{figure}[H]
\centering
\scalebox{.37}{\includegraphics{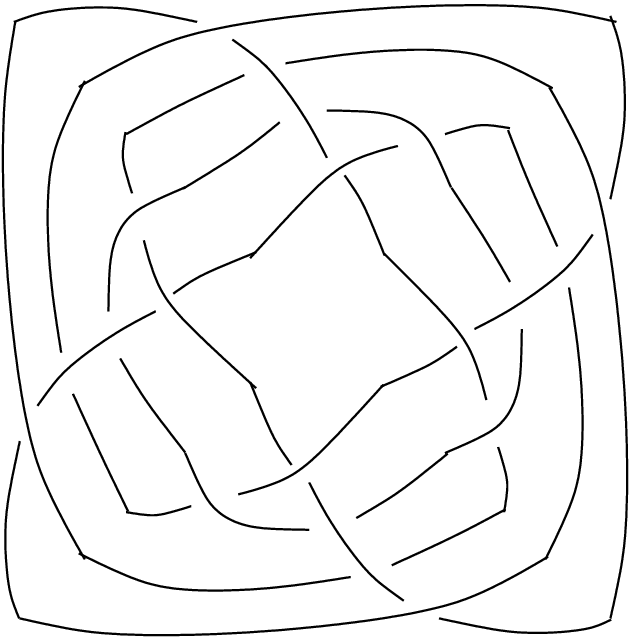}}
\caption{The smallest counterexample to Montesinos-Nakanishi 3-move conjecture, up to mirror image; %$(\sigma_2\sigma_1^{-1}\sigma_2\sigma_3\sigma_4^{-1})^4$; 
$(\sigma_2\sigma_3\sigma_2\sigma_4^{-1}\sigma_1^{-1})^4$
$20$-crossings (the Chen link)}
\label{Chen5braidTrieste}
\end{figure}
%Figure 15

\begin{figure}[H]
\centering
\scalebox{.30}{\includegraphics{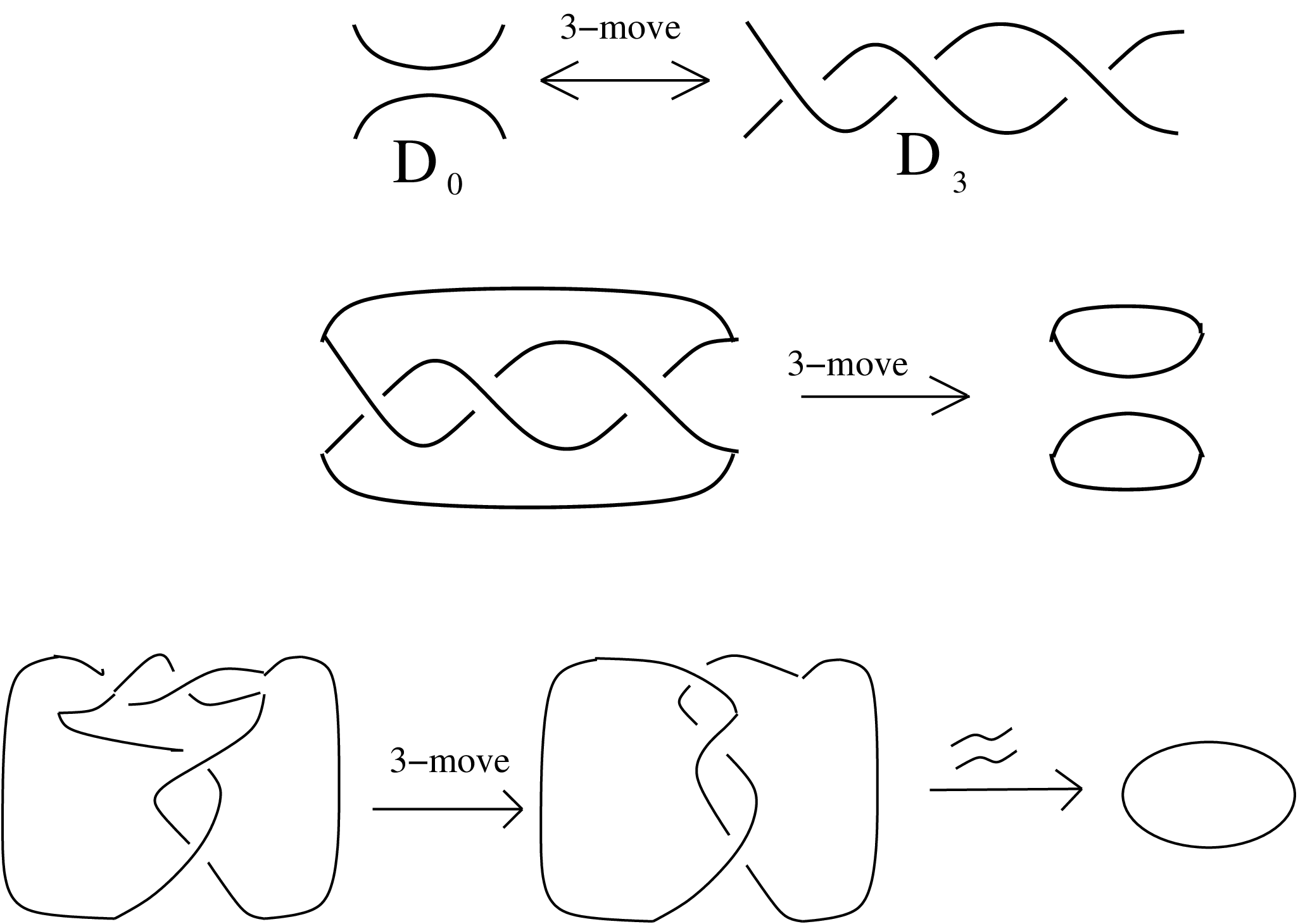}}
\caption{3-move and reduction of the trefoil knot and the figure eight knot to a trivial link.}
\label{3-move-trefoil-feightkn}
\end{figure}
%Figure 16

We found a counterexample by introducing the concept of {\it Burnside groups of links} and showing that the $n$th group is preserved by $n$ moves; see \cite{DaPr1,DaPr2,DaPr3}.

Let us recall first the famous Burnside conjecture concerning groups
\begin{equation*}
B(r,n)=F_r/(w^n),
\end{equation*}
where $F_r$ is the free group on $r$ generators and $w$ is an arbitrary element of $F_r$. 
\begin{problem}\label{Burnside} (Burnside 1902 \cite{Bur}) For what values of $r$ and $n$ is the Burnside group $B(r,n)$ finite?
\end{problem}
Already Burnside proved that for $n=3$ the group $B(r,3)$ is finite \cite{Bur}. Furthermore Levi and van der Waerden \cite{L-W} proved that the group $B(r,3)$ has $3^{n+ \binom{n}{2}+ \binom{n}{3}}$ elements; in particular, $|B(3,3)|=3^7$ and $|B(4,3)|=3^{14}$.\footnote{The Burnside groups of exponent 4 and 6 were proven to be finite by  I. N. Sanov in  1940, \cite{San}, and M. Hall in 1958, \cite{Hall},  respectively.
However, it was proved by Novikov and Adjan in 1968 \cite{N-A-1, N-A-2,
N-A-3} that $B(r,n)$ is infinite whenever $r>1$ and $n$ is an odd and $n\geq
4381$ (this result was later improved by Adjan \cite{Adj}, who showed that $
B(r,n)$ is infinite if $r>1$ and $n$ odd and $n\geq 665$). S. Ivanov proved
that for $k\geq 48$ the group $B(2,2^{k})$ is infinite \cite{Iv}. I. G. Lys{\"e}nok found that $B(2,2^{k})$ is infinite for $k\geq 13$ \cite{Lys}.
It is still an open problem though whether, for example, $B(2,5)$, $B(2,7)$ or $B(2,8)$ are infinite or finite, however if $B(2,5)$ is finite then it has $5^{34}$ elements. M. Vaughan-Lee speculated over 20 years ago (personal information) that when computers are strong enough it can be decided by them whther $B(2,5)$ is infinite or finite. Clearly the time did not come yet (as of October 2024).} \\
The simplest definition of Burnside group of a link, $B_n(L)$, is via the fundamental group of the double branched cover of $S^3$ branched along a link $L$, $\pi_1(M^{(2)}_L)$ that is $B_n(L)=\pi_1(M^{(2)}_L)/(w^n).$
To make this definition useful we interpreted $\pi_1(M^{(2)}_L)$ as a reduced core group of the diagram where core group of the diagram, $\mathrm{Core}(D)$ has very simple 
diagrammatic definition: 
\begin{definition}
 The group, $\mathrm{Core}(L)$ of the link diagram $L$ can be computed from any diagram $D$ of $L$ as follows\footnote{The story is very instructive: around 1976, to Oleg Viro came student and told him that he found new link invariant defined similarly but differently than the fundamental group of link complement. Instead of $c=b^{-1}ab$, used in Wirtinger presentation of the fundamental group of a link complement, he used the relation $c=ba^{-1}b$. The student, named Victor Kobelsky, told Oleg that he checked invariance by Reidemeister move. The prevailing ideology as promoted by Ralph Fox was that knot invariants are some known invariants of algebraic topology related to homotopy or homology groups.
 And, in fact, Oleg thought hard and showed that Victor's invariant is in fact the fundamental group of the double branched cover of $S^3$ branched along the link. Victor never published his funding. A dozen of years later Wada published the result in Topology (see \cite{Wad,Prz5}). As Oleg summarized: {\it take ideas of your students very seriously!}}:
 \begin{enumerate}
     \item[(i)] Generators of the group are indexed by arcs of the diagram (arc is a part of a diagram from a tunnel to the next tunnel);
     \item[(ii)] Relations are given by the crossings of the diagram and they are of the form $ba^{-1}bc^{-1}$, where $a,b,$ and $c$ are generators associated to arcs of the crossing as in Figure \ref{L+Core}.
\end{enumerate} 
\end{definition}
The relation between the core group and $\pi_1(M^{(2)}_D)$ is as follows:
\begin{equation*}
\mathrm{Core}(D)=Z*\pi_1(M^{(2)}_D).
\end{equation*}
\begin{figure}
\centering
\scalebox{.61}{\includegraphics{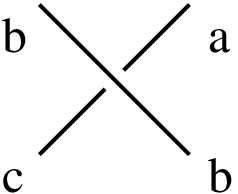}}
\caption{Core group relation at a crossing  $c=ba^{-1}b$}
\label{L+Core}
\end{figure}
%Figure 17

\begin{figure}
\centering
\scalebox{.29}{\includegraphics{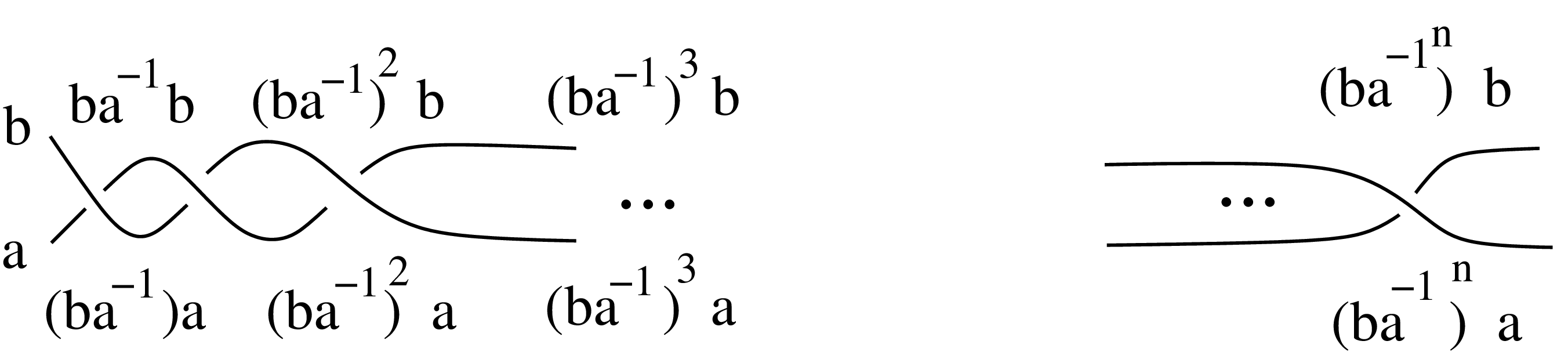}}
\caption{Core group relations applied to $n$-move; $b$ goes to $(ba^{-1})^nb$ and $a$ goes to $(ba^{-1})^na$}
\label{Core-n-move}
\end{figure}
%Figure 18

%\centerline{{\psfig{figure=L+Core.pdf,height=3.1cm}}}
A simple, but breakthrough, observation in my work with Mietek was that the $n$th Burnside group is preserved by $n$-moves (see Figure \ref{Core-n-move} which essentialy contains a ``Proof Without Words"). We checked in \cite{DaPr1} that the Chen link has different the third Burnside group from third Burnside groups of all trivial links ($|B_3(Chen)|=3^{10}$). What was helping us at initial stages of our work was that Mike Newman put the presentation of $B(4,3)$ on web (in package Magma); \cite{New}.\footnote{I gave talk about our work at University of Maryland topology seminar (April 1, 2002) with S.P.~Novikov (1938-2024) in audience. He become sentimental recalling his father P.S.~Novikov work on Burnside problem when he was a young student. He suggested  publishing our findings in Proceedings of the National Academy of Science, which we did in \cite{DaPr2}.}

I am not sure when, instead of core group, I started to consider an effect of $n$ move on a quandle however 
it is expanded in my paper with Mietek \cite{DaPr3} but it is also crucial in my first paper with Maciek and it is even embedded in the title ``Burnside Kei" \cite{NiPr1} (see Subsection \ref{Maciek}).

\subsection{Maciej Niebrzydowski; quandles and their homology}\label{Maciek}
In the Summer of 2003, Maciej (Maciek) Niebrzydowski become my PhD student. Similarly to Mietek he was recommended by Witek Rosicki. He brought with him from 
Poland a deep knowledge (end enthusiasm) of quandles, gave several seminar talks, and convinced me to think about distributive structures and their homology. Our first paper, \cite{NiPr1} was motivated by an analysis of $n$-moves and their effect on quandles. Even  the title of the paper ``Burnside Kei" shows an influence of my previous work with Mietek D{\c a}bkowski. \ 
In the paper we analyze, in particular, action of $n$-moves on Keis (involutive quandles) and general quandles, compare Figure \ref{Quandle-n-move}. Formula for $w_n$ from Figure \ref{Quandle-n-move} can be easily proven by induction on $n$ (by definition $w_{-1}=a, w_0=b$, $w_1=a*b$ and recursively $w_n=w_{n-2}*w_{n-1}$). The inductive step is as follows (here $w_n^{(a)}$ denotes the word obtained from $w_n$ by exchange $a\leftrightarrow b$):
\begin{eqnarray*}
w_n&\stackrel{def}{=}& w_{n-2}*w_{n-1} \stackrel{ind}{=} (w_{n-3}^{(a)}*b)*(w_{n-2}^{(a)}*b) \stackrel{distr}{=}(w_{n-3}^{(a)}*w_{n-2}^{(a)})*b\stackrel{def}{=} w_{n-1}^{(a)}*b, 
\end{eqnarray*}
as needed.

\begin{figure}[ht]
\centering
\scalebox{.28}{\includegraphics{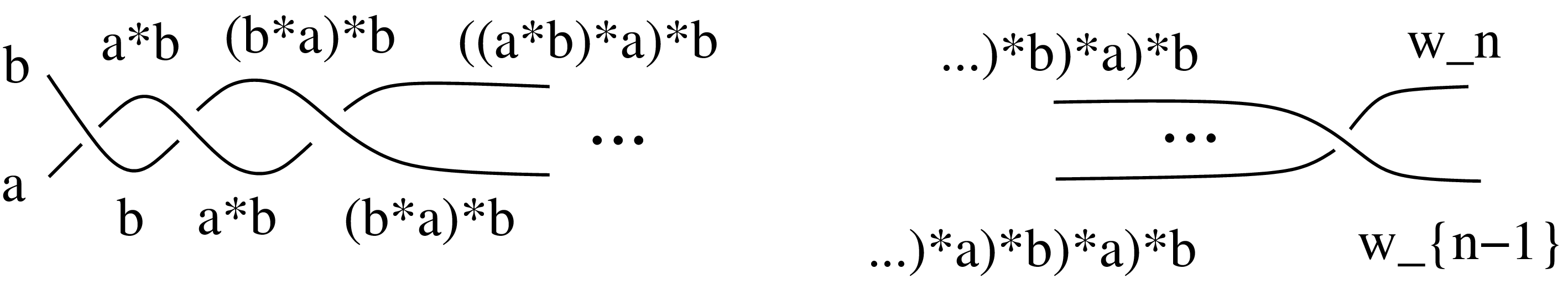}}
\caption{Quandle coloring on $n$-move; $w_n$ denotes the word of length $n+1$ in which $a$ and $b$ alternate and the last letter is $b$, except that $w_0=a$. Note that
$w_2\stackrel{def}{=}b*(a*b)\stackrel{Idem}{=} (b*b)*(a*b) \stackrel{distr}{=} (b*a)*b$}
\label{Quandle-n-move}
\end{figure}

\begin{example} Consider the case of $n=3$. We have $w_2=(b*a)*b$ and $w_3=((a*b)*a)*b$ 
and after closing the 3-move to the trefoil knot we get a presentation of the fundamental qundle of the trefoil:
$$\{a,b\ | \ w_2=a, w_3=b \} = \{a,b\ | \ (b*a)*b=a, (a*b)*a=b \} = \{a,b\ | \ b*a=a\bar *b, a*b=b\bar * a \}.$$
This fundamental quandle is studied in detail in \cite{NiPr3}.\\
Notice that if we assume $*=\bar *$ that is we deal with fundamental kei (involutive quandle), then we get $$\{a,b\ | \ a*b=b*a \}$$ the abelian quandle of two generators, that is the dihedral quandle $R_3=\mathrm{kei}(\mathbb{Z}_3)$.
\end{example}

One of our important contribution was ``Delayed Fibonacci" Conjecture on odd Dihedral quandles \cite{NiPr2,NiPr4}.
\begin{conjecture}
Let $R_k$ be a dihedral quandle, that is the group $\mathbb{Z}_k$ with quandle operation $a*b=2b-a$. If $k$ is odd then it quandle homology satisfies
\begin{equation*}
H_n(R_k)= \mathbb{Z}_k^{f_n},\mbox{for } n>1,
\end{equation*} 
where $f_n$ is the delay Fibonacci sequence, i.e.,
\begin{equation*}
f_1 = f_2 = 0,\, f_3 = 1,\,\,\text{and}\,\, f_n=f_{n-1}+f_{n-3},\,\,\text{for}\,\,n\geq 4.
\end{equation*}
\end{conjecture}
The conjecture was proved in very powerful papers by Clauwens \cite{Cla} and Nosaka \cite{Nos}. Another important result, we obtained with Maciej, was about second homology of odd Takasaki quandle \cite{NiPr5}\footnote{The last calculation in \cite{NiPr5} was made at Oakland airport: in March 2010 after MSRI meeting (Homology theories of knots and links); Misha Khovanov kindly brought me to the airport and I had several hour before my red eye flight so I was able to finish our paper with Maciej.}  In 2018 we generalized 
with Nosaka and my PhD students the result about second homology of odd Takasaki quandle to any quasigroup Alexander quandle \cite{BIMNP}.
\begin{theorem}\label{Shur}
 Let $(X,*)$ be an Alexander quandle ($x*y=tx+(1-t)y$) with $(1-t)$ invertible (equivalently our Alexander quandle is a quasigroup). Then, 
 there is an isomorphism 
 which express the second homology of Alexander quandles in terms of exterior algebras:
 $$H_2^Q(X,\mathbb Z)\cong \frac{X\bigwedge_{\mathbb Z}X}{\{x\wedge y - tx\wedge ty\}_{x,y\in X}}.$$

\end{theorem}
\ \\ 

As long as Maciej was in Lafayette (at University of Louisiana), we worked systematically meeting twice a year (as much as my teaching was allowing). Our last paper is concerning entropic condition and homology of entropic magmas \cite{NiPr6}. This project should be continued as we have a few candidates for ``good" homology of entropic magma. We are led in this by analysis of extentions of entropic magmas. This program should be performed for many other algebraic structures, e.g. identities of Bol-Moufang type \cite{PhVo}.\footnote{Petr Vojtechovsky discussed these identity at one of his talks and 
I was trying then to see homology theory for many of them, but it should be done systematically. {\color{red}Added for e-print: Ania Zamojska-Dzienio visited us at GWU in March-April 2025 and with her and my students 
we advanced homology theory of Bol-Moufang type; still, a lot remains to be done, \cite{CCGOPZ}}.}

\subsection{Carter-Kamada-Saito book}
In my study of distributive structures, after the initial push by Maciek,
the important ingredient was the book by Carter, Kamada, and Saito \cite{CKS} which I studied carefully. It is worth to cite here Matveev from MathSciNet:\\
{\it The last chapter is devoted to a detailed discussion of state sum type invariants derived from quandles (algebraic objects that had been independently introduced in 1982 by D. Joyce \cite{Joy} and the reviewer \cite{Mat}... The authors of the book are among the main founders of this theory and have contributed a great deal to its development. Therefore it is not surprising that this chapter contains many new results. I conclude with the remark that the book may serve as a good introduction for a more or less experienced reader into the beautiful world of knotted surfaces.}
Indeed, one of my projects, which I value a lot, was to generalize results on knotted surfaces into analysis of embeddings of $n$ dimensional manifolds in $n+2$ dimensional manifolds. One of the tools, we used with Witek Rosicki were Roseman results (I have learned from Charlie Frohman that Dennis Roseman work is not yet published so I published it in 3 instalments learning about Roseman moves on the way \cite{Ros1,Ros2,Ros3,PrRo2}).

\subsection{Gda\'{n}sk University, Knots in Poland III, 2010, B{\c e}dlewo} In the summer of 2010 just before Knots in Poland III conference, I visited Witek at University of Gda\'nsk. 
Witek was asking me to find some good problems for his PhD students.
I was looking for something elementary but unusual. The first very useful, most likely not very new, concept was a monoid of binary operations.

%\subsection{Homology of distributive structures.}
Recall that a set $X$ with a binary operation $*: X\times X \to X$ is called a magma. If $*$ is associative $(a*b)*c=a*(b*c)$ then $(X,*)$ is a semigroup and if additionally there is an identity element, $1$ ($a*1=a=1*a$) then $(X,*,1)$ is a monoid.

\begin{definition}\label{Bin1}
For a given set $X$ the set of all binary operations on $X$ is denoted by $Bin(X)$.\footnote{I realized importance of this object in summer 2010 during my visit in Gda\'nsk while trying to understand the significance of rack and quandle homology. I denoted it then $Gd(X)$, see \cite{Prz9}.}
\end{definition}
\begin{proposition}\label{Bin2}
The set $Bin(X)$ forms a monoid with composition $*_1*_2$ of binary operations defined by $a(*_1*_2)b= (a*_1b)*_2b$ and the identity element 
$*_0$ defined by $a*_0b=a$.
\end{proposition}
\begin{proof} We check quickly that 
\begin{equation*}
a(**_0)b= (a*b)*_0b=a*b= (a*_0b)*b=a(*_0*)b,
\end{equation*}
thus $**_0=*=*_0*$. Associativity follows from associativity of composition of functions, we have:
\begin{equation*}
a(*_1*_2)*_3b=(a*_1*_2b)*_3b=((a*_1b)*_2b)*_3b=(a*_1b)*_2*_3b=a*_1(*_2*_3)b.  
\end{equation*}
Thus $(*_1*_2)*_3= *_1(*_2*_3)$, as needed.
\end{proof}

The monoid $Bin(X)$ was probably not new\footnote{The diagonal map $\Delta :X\rightarrow X\times X$ given by $\Delta(a)=a\times a$ can be regarded as a very special comultiplication and generalized (e.g., as used in some unpublished work of Berfriend Fauser).}. A binary operation $\ast $ on a set $X$ is (right) self-distributive if 
\begin{equation*}
(x\ast y)\ast z = (x\ast z)\ast (y\ast z)
\end{equation*}
for every $x,y,z\in X$. A magma with a right self-distributive operation is called a shelf (racks and quandles are examples of shelves). Fenn, Rourke, and Sanderson developed homology theory for racks.

A multi-shelf is a set $X$ with a collection of binary operations $*_i$, $i\in I$, that are mutually right self-distributive, i.e., $(x*_iy)*_jz=(x*_jz)*_i(y*_jz)$ for all $i,j\in I$.

%\subsection{Gdansk Lectures} 
From Fall of 2010 I started teaching knot theory and its ramifications at University of Gda\'nsk. I was convinced to do so by my Gda\'nsk friend Witold Rosicki. 
Notes made by Micha{\l} Jab{\l}onowski (Witek's student), were published by Gda\'nsk University press \cite{Prz9} (in Polish). My life in Gda\'nsk was relaxing and productive thanks to  Witek (my formal association with University of Gda\'nsk covers the period 2010-2025).

%\section{From Fox coloring to Yang-Baxter coloring}

\section{Short introduction to homology theory} 

\subsection{Homology from chain complexes}
Let $k$ be a commutative ring with identity (e.g. $k= \mathbb Z$). Recall that 
a chain complex 
$\mathcal C = \{C_n,\partial_n \}_{n\in \mathbb Z}$ 
is a sequence of 
$k$-modules (abelian groups for $k= \mathbb Z$) $C_n$, and homomorphisms $\partial_n: C_n \to C_{n-1}$ such that, $\partial_{n} \partial_{n+1} =0$ for any  $n$ (we write succinctly $\partial^2 = 0$). One defines $n$th homology of a chain complex as
\begin{equation*}
H_{n}(C)=\ker \left( \partial _{n}\right) /\mathrm{im}\left( \partial
_{n+1}\right).
\end{equation*}

\subsection{Chain complex from abstract simplicial complex}

\begin{definition}\label{ASC}
An abstract simplicial complex $\mathcal{K}=(V,S)$ consists of a set $V=V(\mathcal{K})$ of vertices and a set $S = S(\mathcal{K})$ of finite (nonempty) subsets of $V$
called simplexes (or faces) of $\mathcal{K}$ which is closed under inclusion (except for $\emptyset $) and any vertex is a ($0$-dimensional) simplex. A subset $\left\{ v_{i_{0}},v_{i_{1}},\ldots ,v_{i_{n}}\right\}\in S$ of $\left(
n+1\right) $ different vertices is called an $n$-dimensional simplex ($n$-simplex) and the dimension $\dim (\mathcal{K})$ of $\mathcal{K}$ is the supremum of dimensions of its simplices (so $\dim (\mathcal{K})$ can also be $\infty $).
\end{definition}
To every abstract simplicial complex with ordered vertices we associate the chain complex as follows. Let $X_n$ be the set of simplices of dimension $n$. We define $C_n=kX_n$. The map $\partial_n: C_n \to C_{n-1}$ is defined on bases $X_n$ by 
\begin{equation*}
\partial_{n}(v_{0},\,v_{1},\,\ldots,v_{n})=\sum_{i=0}^{n}(-1)^{i}(v_{0},\,\ldots
,v_{i-1},\,v_{i+1},\,\ldots ,v_{n})
\end{equation*}
where $(v_0,\,v_1,\,\ldots,v_n)\in X_n$ and $v_0 < v_1 < \cdots < v_n$ in chosen ordering of $V$, and, as usually, extended $k$-linearly to the map on $C_{n}$.
Homology of $\mathcal K$ is defined to be homology of this chain complex.

In the next subsection we generalize abstract simplicial complex to presimplicial set.
%\subsection{Geometric realization of a presimplicial set}
A presimplicial set keeps an instruction how to construct a topological space (CW complex) build of simplexes (A.~Hatcher \cite{Hat} calls it $\Delta$-set).\footnote{
During my undergraduate studies at the University of Warsaw, I was rather well educated in homological algebra (my master degree advisors in 1977 were algebraic topologists Agnieszka Bojanowska and Stefan Jackowski).
However, after 28 years while I was discovering connection between Khovanov homology and Hochschild homology \cite{Prz7}, I had to learn it again. I bought Loday book \cite{Lod} and I was carrying it with me all the time. I forgot that in the Spring of 1984 Loday gave a course on cyclic and Hochschild homology at University of Warsaw and I was taking notes. What is written below, and applied to quandle homology was motivated by Loday's book. A more comprehensive account of my thinking and findings is published in \cite{Prz8}. I wrote there: {\it This paper is a summary of numerous talks I gave last year: from my Summer 2010 talk at Knots in Poland to a seminar at Warsaw Technical University in June 2011...} It was after my talk at the Warsaw University of Technology when I was suggested to publish my findings in \emph{Demonstratio Mathematicae}. The follow up paper appeared in \cite{Prz11} (see also my Gda\'nsk book \cite{Prz9} where some results were published in Polish before journal publications in English.) {\color{red} Added for e-print: It was Basia Roszkowska who invited me to give a talk at} {\color{red} the Warsaw University of Technology. Basia did her Master Degree at Warsaw University with Stefan Jackowski as an advisor.}}

\subsection{Presimplicial category (set, module); geometric realization}\ 

For completeness we recall here the notion of a presimplicial category (in the case of sets, or $k$-modules) and the geometric realization of a presimplicial set.

\begin{definition}\cite{E-Z,Lod,Prz8}\label{Presimplicial module}
\begin{enumerate}
\item[(1)]  A presimplicial set $\mathcal X = (X_n,\,d_i)$  is a collection of sets $X_{n}$, $n \geq 0$, together with maps called face maps or face operators,
\begin{equation*}
d_{i,n}=d_{i}: X_{n} \to X_{n-1},\,i=0,\ldots,n
\end{equation*}
such that
\begin{equation*}
d_{i}d_{j} = d_{j-1}d_{i} \mbox{ for } 0 \leq i<j \leq n.
\end{equation*}
\item[(2)] A presimplicial module $\mathcal C$ is a collection of $k$-modules $C_{n}$, $n \geq 0$, together with maps called face maps or face operators,
\begin{equation*}
d_{i}: C_{n} \to C_{n-1},\,i=0,\ldots,n   
\end{equation*}
such that,
\begin{equation*}
d_{i}d_{j} = d_{j-1}d_{i} \mbox{ for } 0 \leq i<j \leq n.
\end{equation*}
\item[(3)] If $(X_n,\,d_i)$ is a presimplicial set than we have the naturally associated presimplicial module by choosing $C_n=kX_n$ (the free
k-module with basis $X_n$) and $d_{i,n}$ being extensions of $d_{i,n}: X_n \to X_{n-1}$ to the map $C_n\to C_{n-1}$.
\end{enumerate}
\end{definition}
Recall that a presimplicial module induces the chain complex $(C_n,\,\partial_n)$, where the boundary operator is defined as $\partial_n=\sum_{i=0}^n(-1)^id_i$. This chain complex leads to homology (and cohomology) theory of the presimplicial module.

\begin{definition}
A presimplicial set allows standard geometric
realization: We think of $X_{n}$ as labels for simplexes and face maps $%
d_{i} $ as gluing instructions. For a more precise definition we need a \textquotedblleft model\textquotedblright\ simplex
\begin{equation*}
\Delta ^{n}= \{ (x_{0},\ldots ,x_{n})\in \mathbb{R}^{n+1}\mid \sum_{i=0}^{n}x_{i}=1,\,x_{i}\geq 0\}
\end{equation*}
and embedding maps
\begin{equation*}
d^{i,n-1}= d^{i}:\Delta ^{n-1}\rightarrow \Delta^{n}
\end{equation*}
given by 
\begin{equation*}
d^{i}\left( x_{0},\ldots ,x_{i-1}\right) =\left( x_{0},\ldots,x_{i-1},0,x_{i},\ldots,x_{n-1}\right) .
\end{equation*}
Recall, that a geometric realization $\left\vert \mathcal{X}\right\vert $ of $\mathcal{X}$ is a CW complex defined by
\begin{equation*}
\left\vert \mathcal{X}\right\vert =\bigcup\limits_{n\geq 0}\left(
X_{n}\times \Delta ^{n}\right) /\sim ,
\end{equation*}
where relation $\sim $ is given by%
\begin{equation*}
\left( d_{i}(a),t\right) \sim \left( a,d^{i}(t)\right) ,
\end{equation*}
for $a\in X_{n},$ $t\in \Delta ^{n-1}$.
\end{definition}
Our geometric realization $\left\vert \mathcal{X}\right\vert $ is a CW-complex, however in fact, it has a much simpler structure as it is glued from simplexes, i.e., it is $\Delta $-complex (see Hatcher \cite{Hat}, Section 2.1). A basic first example of a presimplicial set and its geometric realization is given by abstract and geometric simplicial complexes (see Definition~\ref{ASC}). 

\begin{example}\label{Example 13.2}
 Let ${\mathcal X}$ be an abstract simplicial complex $X=(V,S)$. If we order its vertices then
${\mathcal X}$ is a presimplicial set with $X_n$ being the set of $n$ simplexes
of ${\mathcal X}$ and face maps are defined on each simplex in a standard way 
\begin{equation*}
d_i(x_0,\ldots,x_n)=
(x_0,\ldots,x_{i-1},x_{i+1},\ldots,x_n).    
\end{equation*}
The geometric realization of the above presimplicial set is a classical geometric simplicial complex.
\end{example}
The other two important examples of a presimplicial set are
related to semigroup (associative) homology and (one term) distributive homology, respectively:
\begin{definition}\label{preass-distr}
Let $(X,*)$ be a magma and let $X_n=X^{n+1}$. We define two presimplicial sets $\mathcal X^{assoc}$ and $\mathcal X^{distr}$ with $d_{i,n}:X_n \to X_{n-1}$ depending on $*$ being associative or right self-distributive, respectively.
\begin{enumerate}
\item[(1)] Let $(X,*)$ be a semi-group and 
\begin{equation*}
d_i^{ass}(x_0,x_1,\ldots,x_n)=(x_0,\ldots,x_{i-2},x_{i-1}*x_i,x_{i+1},\ldots,x_n)
\end{equation*}
(in particular $d_0^{ass}(x_0,x_1,\ldots,x_n)=(x_1,\ldots,x_n)$). Then $\mathcal X^{assoc}$ is a presimplicial set.
%\footnote{One checks the condition $d_id_j=d_{j-1}d_i$ for $i<j$ directly. The interesting case is for 
%$j=i+1$; then we have: $d_id_{i+1}((x_0,x_1,\ldots,x_n)=d_i(x_0,\ldots,x_{i-1},x_i*x_{i+1},x_{i+2},\ldots,x_n)=(x_0,\ldots,x_{i-2},x_{i-1}*(x_i*x_{i+1}),x_{i+2},\ldots,x_n)$. Also $d_id_i(x_0,x_1,\ldots,x_n)=d_i(x_0,\ldots,x_{i-2},x_{i-1}*x_i,x_{i+1},\ldots,x_n)= (x_0,\ldots,x_{i-2},(x_{i-1}*x_i)*x_{i+1},x_{i+2},\ldots,x_n)$ so the same by associativity. }
\item[(2)] Let $(X,*)$ be a shelf, that is $(a*b)*c=(a*c)*(b*c)$ and
\begin{equation*}
d_i^{distr}(x_0,x_1,\ldots,x_n)=(x_0*x_i,\ldots,x_{i-1}*x_i,x_{i+1},x_{i+2},\ldots,x_n)
\end{equation*}
(in particular $d_0^{distr}(x_0,x_1,\ldots,x_n)=(x_1,\ldots,x_n)$). 
Then $\mathcal X^{distr}$ is a presimplicial set.%\footnote{One checks the condition $d_id_j=d_{j-1}d_i$ for $i<j$ directly. We have $d_id_{j}((x_0,x_1,\ldots,x_n)=d_i(x_0*x_j,\ldots,x_i*x_j,\ldots,x_{j-1}*x_j,x_{j+1},\ldots,x_n)=((x_0*x_j)*(x_i*x_j),...,(x_{i-1}*x_j)*(x_i*x_j),x_{i+1}*x_j,...,x_{j-1}*x_j,x_{j+1},...,x_n)$. Also $d_{j-1}d_i(x_0,x_1,...,x_n)=d_{j-1}(x_0*x_i,...,x_{i-1}*x_i,x_{i+1},...,x_n)= ((x_0*x_i)*x_j,...,(x_{i-1}*x_i)*x_j,x_{i+1}*x_j,...,x_{j-1}*x_j,x_{j+1},...,x_n)$ so the same by distributivity. }
\item[(3)] The first (1) presimplicial set leads to semigroup homology and the second (2) leads to (one term) distributive homology (defined originally in \cite{Prz8}).
\end{enumerate}
\end{definition}

\begin{remark}
\label{rem:Cond_1}
Condition
\begin{equation*}
d_{i}d_{j}=d_{j-1}d_{i}
\end{equation*}%
for $i<j$ in (1) of the Definition~\ref{preass-distr} can be verified directly. The interesting case is when $j=i+1$, i.e.,
\begin{eqnarray*}
d_{i}d_{i+1}\left( x_{0},x_{1},\ldots ,x_{n}\right) &=&d_{i}(x_{0},\ldots
,x_{i-1},x_{i}\ast x_{i+1},x_{i+2},\ldots ,x_{n}) \\
&=&(x_{0},\ldots ,x_{i-2},x_{i-1}\ast (x_{i}\ast x_{i+1}),x_{i+2},\ldots
,x_{n}).
\end{eqnarray*}
and also 
\begin{eqnarray*}
d_{i}d_{i}(x_{0},x_{1},\ldots ,x_{n}) &=&d_{i}\left( x_{0},\ldots
,x_{i-2},x_{i-1}\ast x_{i},x_{i+1},\ldots ,x_{n}\right) \\
&=&(x_{0},\ldots ,x_{i-2},(x_{i-1}\ast x_{i})\ast x_{i+1},x_{i+2},\ldots
,x_{n})
\end{eqnarray*}
so, by associativity $d_{i}d_{i+1}=d_{i}d_{i}$.

Condition (2) in (ii) of the Definition~\ref{preass-distr}, i.e., 
\begin{equation*}
d_{i}d_{j}=d_{j-1}d_{i}\text{ for }i<j
\end{equation*}%
can also be verified easily. In particular 
\begin{equation*}
d_{i}d_{j}(x_{0},x_{1},\ldots ,x_{n})=d_{i}(x_{0}\ast x_{j},\ldots
,x_{i}\ast x_{j},\ldots ,x_{j-1}\ast x_{j},x_{j+1},\ldots ,x_{n})
\end{equation*}
\begin{equation*}
=((x_{0}\ast x_{j})\ast (x_{i}\ast x_{j}),\ldots ,(x_{i-1}\ast x_{j})\ast
(x_{i}\ast x_{j}),x_{i+1}\ast x_{j},\ldots ,x_{j-1}\ast x_{j},x_{j+1},\ldots
,x_{n})
\end{equation*}
and on the other hand 
\begin{equation*}
d_{j-1}d_{i}(x_{0},x_{1},\ldots ,x_{n})=d_{j-1}(x_{0}\ast x_{i},\ldots
,x_{i-1}\ast x_{i},x_{i+1},\ldots ,x_{n})
\end{equation*}
\begin{equation*}
=((x_{0}\ast x_{i})\ast x_{j},\ldots ,(x_{i-1}\ast x_{i})\ast
x_{j},x_{i+1}\ast x_{j},\ldots ,x_{j-1}\ast x_{j},x_{j+1},\ldots ,x_{n}),
\end{equation*}
so by distributivity $d_{i}d_{j}=d_{j-1}d_{i}$ as claimed.
\end{remark}

\subsection{Precubic sets and modules; geometric realization. }

We define the precubic set in a way similar to presimplicial set but replacing simplexes by cubes. Also, analoguesly like in the case of simplicial set we use precubic set as an instruction how to glue cubes to get a topological space (CW-complex). We call this a geometric realization of a precubic set and it has the same homology as a precubic set.

\begin{definition}\label{precubic}
\begin{enumerate}
\item[(1)] A precubic set ${\mathcal C}=(X_n,\,d_{i,n}^{\varepsilon})$ is 
a collection of sets $X_n$, $n\geq 0$, together  with maps called face maps
\begin{equation*}
d_{i,n}^{\varepsilon} = d_{i}^{\varepsilon}: X_n\to X_{n-1},\, i = 1,\ldots, n,\, \varepsilon \in\{0,\,1\}
\end{equation*}
such that 
\begin{equation*}
d_i^{\varepsilon}d_j^{\delta}= d_{j-1}^{\delta}d_i^{\varepsilon} \mbox{ for } 1\leq i < j \leq n,\ \varepsilon,\delta \in \{0,1\}.
\end{equation*}
\item[(2)] A precubic module $\mathcal C$ is a collection of $k$-modules $C_{n}$, $n \geq 0$, together with maps called face maps or face operators,
\begin{equation*}
d_{i,n}^{\varepsilon}=d_{i}^{\varepsilon}: C_{n} \to C_{n-1},\, i = 1,\ldots, n,\, \varepsilon\in \{0,\,1\}  
\end{equation*}
such that
\begin{equation*}
d_i^{\varepsilon}d_j^{\delta} = d_{j-1}^{\delta}d_i^{\varepsilon} \mbox{ for } 1 \leq i<j \leq n, \ \varepsilon,\delta \in \{0,1\}.  
\end{equation*}
\item[(3)] For a precubic set $(X_{n},d_{i}^{\epsilon })$ there is a naturally associated precubic module obtained by choosing a free $k$-module $C_{n}=kX_{n}$ with basis $X_{n}$ and $d_{i,n}^{\epsilon }:C_{n}\rightarrow C_{n-1}$ being extensions of $d_{i,n}^{\epsilon }:X_{n}\rightarrow X_{n-1}$.
\end{enumerate}
\end{definition}
Recall that a precubic module leads to the chain complex $(C_n,\partial_n)$ by putting 
\begin{equation*}
\partial_n=\sum_{i=1}^n(-1)^i(d_i^{(0)}-d_i^{(1)}).    
\end{equation*}
This chain complex leads to homology (and cohomology) theory.

\begin{definition}\label{Definition 13.6}
The geometric realization of a precubic set is a topological space (CW-complex) defined as follows (notice that
$X_n$ is indexing cubes and precubic structure gives an instruction how to glue the cubes together):
\begin{equation*}
|{\mathcal X}|= \bigcup_{n\geq 0} (X_n \times I^n)/\sim_{rel},  
\end{equation*}
where $\sim_{rel}$ is an equivalence relation generated by $(x,d^i_{\varepsilon}(y)) \sim_{rel} (d_i^{\varepsilon}(x),y)$, $y\in I^{n-1}$ and as before,  $d_i^{\varepsilon}: X_i\to X_{i-1}$, $x\in X_n$, and the model cube 
\begin{equation*}
I^n = [0,1]^n = \{(x_1,\ldots,x_n)\in R^n \mid 0\leq x_i \leq 1 \}, 
\end{equation*}
$d^i_\varepsilon = d^{i,n-1}_{\varepsilon}: I^{n-1}\to I^n$, where $d^{i,n-1}_{\varepsilon}$ is defined by
\begin{equation*}
d^{i,n-1}_{\varepsilon}(x_{1},\ldots,x_{n-1}) = (x_{1},\ldots,x_{i-1},\varepsilon,x_{i},\ldots,x_{n-1}).     
\end{equation*}
\end{definition}

\subsection{Invariants of links from Yang-Baxter equations}

``One can trace basically three streams of ideas from which YBE (Yang-Baxter Equation) has emerged:
the Bethe Ansatz, commuting transfer matrices in statistical mechanics, and factorize $S$ matrices in field theory" \cite{Jimb}.

Homology for racks (i.e., invertible distributive magmas)
was defined sometimes between $1990$ and $1995$ by Fenn, Rourke, and Sanderson \cite{FRS-1,FRS-2,FRS-3,Fenn}. 
Cocycle invariants of knots were constructed by Carter et al.
(i.e., for quandles in \cite{CJKS,CJKLS} and for biquandles/set theoretic
Yang-Baxter operators in \cite{CES}). I introduced one term homology of Yang-Baxter operator with a (column) unital condition during conference in Moscow and for two terms at the
conference in Oberwolfach (2012). The homology of general Yang-Baxter operators were introduced in 2012 in \cite{Leb-1,Leb-2,Prz10,Prz11} (compare also
\cite{FIKM}, and \cite{Eis-1,Eis-2}).

%\subsection{Invariants of arc colorings}

\subsection{Distributive homology}
We recall, after \cite{Prz8,Prz11} basic facts about distributive homology and its relation to presimplicial sets and precubic sets.

A shelf (or right distributive system (RDS)) is a binary structure $\left( X;\ast \right) $, where $X$ is a set and $\ast :X\times X\rightarrow X$ is a right self-distributive binary operation (i.e., $(a\ast b)\ast c=(a\ast c)\ast (b\ast c)$). For simplicity, we will work with chain complexes and homology over $\mathbb{Z}$ however, we can replace $\mathbb{Z}$ by any commutative ring $k$ in our considerations.

We start our discussion from atomic definition, one term distributive homology,  introduced in 2010 just before Knots in Poland III conference \cite{Prz8}.

\subsection{One-term distributive homology}

\begin{definition}\label{Definition 6.1}
For a shelf $\left( X;\ast \right) ,$ a (one-term)
distributive chain complex $\mathcal{C}^{(\ast )}$ consists of chain groups $%
C_{n}=\mathbb{Z}X^{n+1}$ and boundary operations $\partial _{n}^{(\ast )}:C_{n}\rightarrow
C_{n-1}$ given on the basis $X^{n+1}$ by:
\begin{equation*}
\partial _{n}^{(\ast )}(x_{0},\ldots ,x_{n})=(x_{1},\ldots
,x_{n})+\sum_{i=1}^{n}(-1)^{i}(x_{0}\ast x_{i},\ldots ,x_{i-1}\ast
x_{i},x_{i+1},\ldots ,x_{n}).
\end{equation*}
The homology $H_{n}^{(\ast )}(X)$ of this chain complex is called a \emph{
one-term distributive homology} of $(X;\ast )$.
\end{definition}
We directly verify that $\partial^{(*)}\partial^{(*)}=0$. Furthermore, if we put $C_{-1}=\mathbb{Z}$ and $\partial_{0}(x)=1$, then $\partial_{0}\partial_{1}^{(\ast)}=0$. Thus, we obtain an augmented distributive chain complex and an augmented (one-term) distributive homology, $\widetilde{H}_{n}^{(\ast )}$. As in the classical case we get:
\begin{proposition}\label{Proposition 6.4}
\begin{equation*}
 H_n^{(*)}(X)=
 \begin{cases}
 \mathbb Z \oplus \tilde H^{(*)}_n(X) & n = 0 \\
 \tilde H^{(*)}_n(X) & \text{otherwise.}
 \end{cases}   
\end{equation*}
\end{proposition}

If  $(X;*)$ is a rack then the complex $(C^{(*)}_n, \partial^{(*)})$ is acyclic, but in the general case
of a shelf or spindle homology can be nontrivial with nontrivial free and torsion parts (joint work with
A.~Crans, K.~Putyra and A.~Sikora \cite{CPP,PrPu1,P-S}).

If we define $d_i:C_n \to C_{n-1}$, $0\leq i \leq n$,  by
\begin{equation*}
d_i(x_0,\ldots,x_n)=(x_0*x_i,\ldots,x_{i-1}*x_i,x_{i+1},\ldots,x_n),
\end{equation*}
then $(X^{n+1},d_i)$  is presimplicial set and $(C_n,d_i)$ is a presimplicial module  .
If we define degeneracy maps
\begin{equation*}
s_i(x_0,\ldots,x_n)=(x_0,\ldots,x_{i-1},x_i,x_i,x_{i+1},\ldots,x_n)
\end{equation*}
then one checks that $(C_n,\,d_i,\,s_i)$ is a very weak simplicial module, see \cite{Prz8}.\footnote{
It was a very important impulse for my research when I decided to
explore to what extent  a quandle presimplicial 
module is also a simplicial module with this degeneracy; this, in particular, led me to the definition of a weak and very
weak simplicial module. Simplicial module under the name complete semi-simplicial complex was introduced by Eilenberg and Zilber in 1950, \cite{E-Z}.} If we assume idempotency, that is $(X;*)$ is a
spindle, then $(C_n,\,d_i,\,s_i)$ is a weak simplicial module and the degenerate part $(C^D_n,\,\partial_n)$
is a subchain complex which splits from $(C_n,\,\partial_n)$ (see \cite{Prz8}). This split is analogous to the one
conjectured in \cite{CJKS} and proved in \cite{L-N} for classical quandle homology (for a history of
quandle homology see \cite{Car}). In \cite{NiPr2} we
gave very short, easy to visualize and to generalize, proof using the split map $C^{N}_{n} \to C_{n}$ given by
$(x_0,x_1,\ldots,x_n) \to (x_0,x_1-x_0,\ldots,x_n-x_{n-1})$, essentially following \cite{L-N}.\footnote{$C^N_n$ denotes the normalized chain complex, that is the quotient of $C_n$ by the degenerate complex $C^D_n$.} The substantial generalization is given in my paper with Krzysztof Putyra \cite{PrPu2}. We showed that in the case of quandles the rack homology follows from the quandle homology in a form similar to K\"uneth formula (see Subsection \ref{Putyra}).

We can essentially repeat our definitions given above in
the case when $(X;\ast )$ is a shelf and $Y$ is a shelf-set, i.e., $\ast
:Y\times X\rightarrow Y$ with $(y\ast x_{1})\ast x_{2}=(y\ast x_{2})\ast
(x_{1}\ast x_{2})$ (see Figure $20$ for a visualization).
The presimplicial set $(Y\times X^{n+1},d_i)$ has face maps $d_i$ defined by
\begin{equation*}
d_i(y,x_0,\ldots,x_n)=(y*x_i,x_0*x_i,\ldots,x_{i-1}*x_i,x_{i+1},\ldots,x_n).  
\end{equation*}
The face map $d_i$ is visualized in Figure \ref{di-distrNew} and this graphical representation of $d_{i}$ will play an important role when we generalize distributive homology to Yang-Baxter homology.

\begin{figure}[ht]
\centering
\scalebox{.40}{\includegraphics{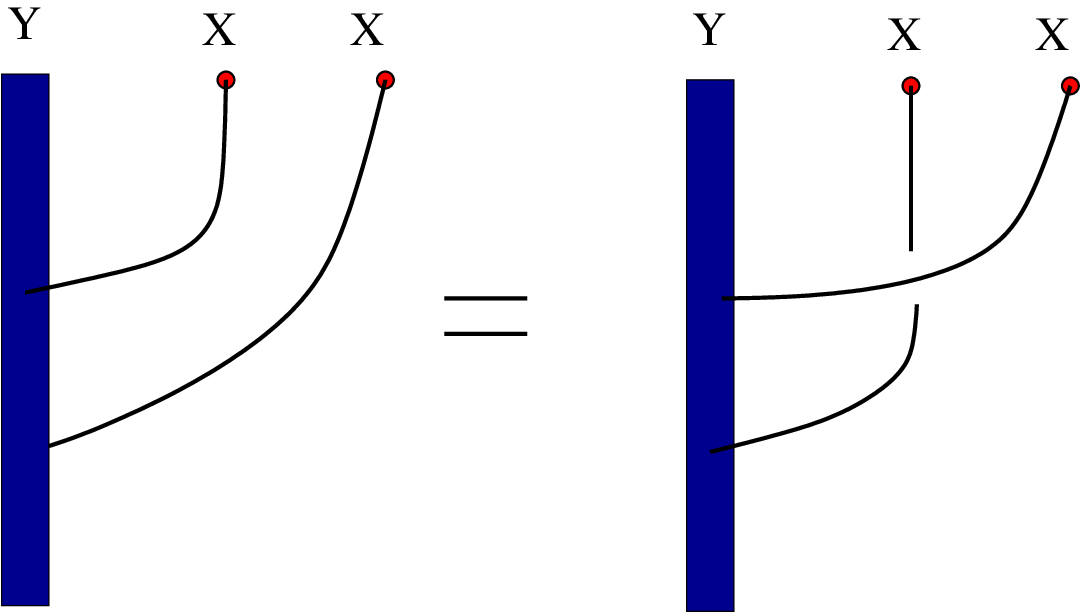}}
\caption{Graphical interpretation of the axiom for $X$-shelf-set $Y$; $(y*x_1)*x_2 = (y*x_2)*(x_1*x_2)$}
\label{Y-shelf-axiom}
\end{figure}
\ \\
%\centerline{\psfig{figure=Y-shelf-axiom.eps,height=3.3cm}}\centerline{Figure 2.1; Graphical interpretation of the axiom for $X$-shelf-set $Y$} \centerline{$(y*x_1)*x_2 = (y*x_2)*(x_1*x_2)$}

\begin{figure}
\centering
\scalebox{.40}{\includegraphics{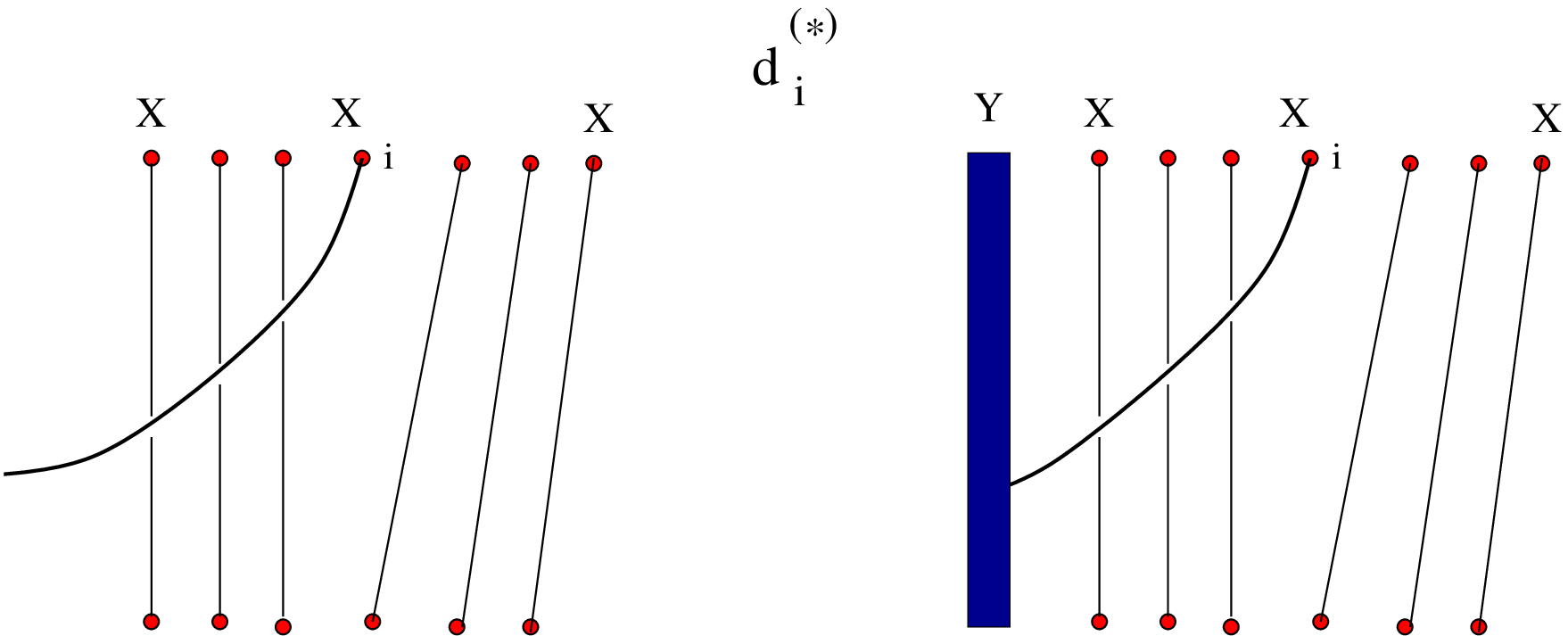}}
\caption{Graphical interpretation of a one term face map $d^{(*)}_i$}
\label{di-distrNew}
\end{figure}
\ \\
%\centerline{\psfig{figure=di-distrNew.eps,height=4.3cm}}
%\centerline{Figure 2.2; Graphical interpretation of a one term face map $d^{(*)}_i$}

\subsection{Multi-term distributive homology}\label{Subsection 2.2}

The first homology theory related to a self-distributive structure was constructed in early 1990s by
Fenn, Rourke, and Sanderson \cite{FRS-2} and motivated by (higher dimensional) knot theory\footnote{ Roger Fenn, \cite{Fenn}, states:
"Unusually in the history of mathematics, the discovery of the homology and classifying space of a rack can be precisely dated to 2 April 1990."}.
For a rack $(X,*)$, they defined rack homology $H_n^R(X)$ by taking $C^R_n=\mathbb Z X^n$ and $\partial_n^R: C_n \to C_{n-1}$ is given by 
$\partial_n^R = \partial_{n-1}^{(*)}-\partial_{n-1}^{(*_0)}$.
Our notation has grading shifted by 1, that is\footnote{A difference in conventions follow from the fact that we go for and back from presimplicial to precubic sets or modules.}, 
\begin{equation*}
C_n(X)= C^R_{n+1}= \mathbb Z X^{n+1}.
\end{equation*}
It is  routine to
check that $\partial^R_{n-1}\partial_n^R=0$. However, it is an interesting question what properties
of $*_0$ and $*$ are really used. 
As one of the first observations made in \cite{Prz8}, I noticed that it
is distributivity again which makes $(C^{R}(X),\partial _{n}^{R})$ into a chain complex. More generally we
observed that if $*_1$ and $*_2$ are right self-distributive and distributive  with respect to each other,
then $\partial^{(a_1,a_2)}= a_1\partial^{(*_1)}+a_2\partial^{(*_2)}$ leads to a chain complex
(i.e., $\partial^{(a_1,a_2)}\partial^{(a_1,a_2)}=0$).\\
We can repeat this construction for any number of pairwise distributive operations $*_1,\,*_2,\,\ldots,\,*_k$ and consider $kX^{n+1}$ with a boundary map defined by 
\begin{equation*}
\partial^{(a_1,\,a_2,\ldots,a_k)}= \sum_{i=1}^ka_i\partial^{(*_i)}   
\end{equation*}
(see \cite{Prz8,Prz11} for details). 

\subsection{Work with Krzysztof Putyra}\label{Putyra}
One of our important and rather unexpected contribution is theorem
that rack homology (and degenerate homology) follows from quandle homology. The relation is very convincing and resembles K\"uneth formula.
\begin{theorem}[\cite{PrPu2}] Let $X$ be a quandle and $k$ a Principal Ideals Domain (e.g., $\mathbb{Z}$) then there is a short exact sequence of homology
\begin{equation*}
0 \to \bigoplus_{p+q=n}\hat H_{q-2}(X) \otimes H^N_p(X) \to H^D_n(X) \to \bigoplus_{p+q=n-1}Tor(\hat H_{q-2}(X), H^N_p)(X)\to 0    
\end{equation*}
which splits. In the formula $\hat{H}$ denotes augmented rack homology, $H^N$ normalized (quandle) homology, and $H^D$ degenerate homology.
\end{theorem}

We also computed with Krzysztof various multiterm homology, including that for finite distributive lattices, including Boolean algebras \cite{PrPu1}.\footnote{While working with Krzysztof on this paper (August 23, 2011), there was an earthquake and we both got a bit shaken. However, after spending two years in California (University of California at Riverside, 1990-1992) I was already quite used to earthquakes by then.}  Since examples considered in \cite{PrPu1} are interesting and also
asking for more work to be done, let us recall the following definition.

\begin{definition}\label{Lattice} Let $(X,\ast _{\cup },\ast _{\cap })$ be a distributive lattice and let $\ast _{\sim }=\ast _{\cap }\ast _{\cup },$
i.e., 
\begin{equation*}
a\ast _{\sim }b=b.
\end{equation*}%
As before, let $\ast _{0}$ be the identity element of $Bin(X)$, that is, $a\ast
_{0}b=a.$ With these choices, the chain complex with the four term boundary
operation $\partial $ on $kX^{n+1}$ is defined by: 
\begin{equation*}
\partial^{(a,b,c,d)}=a\partial ^{(\ast _{0})}+b\partial ^{(\ast _{\cup
})}+c\partial ^{(\ast _{\cap })}+d\partial^{(\ast _{\sim })}.
\end{equation*}
\end{definition}
In \cite{PrPu2} we focused on the four-element Boolean algebra (i.e.,
subsets of a two-element set).

\section{From distributive homology to Yang-Baxter homology}\label{Section 3}

We can extend our basic construction described in the introduction
(using still a very naive point of view), as follows: Fix a finite set $X$ and color semi-arcs of $D$ (parts of $D$ from a crossing to a crossing) by elements of $X$
 allowing different weights from
some ring $k$ for every crossing (following statistical mechanics terminology we call these weights Boltzmann weights).
We also allow to distinguish between a negative and a positive
crossing (see Figure~\ref{Boltzman-weightWW}).
\begin{figure}[ht]
\centering
\scalebox{.41}{\includegraphics{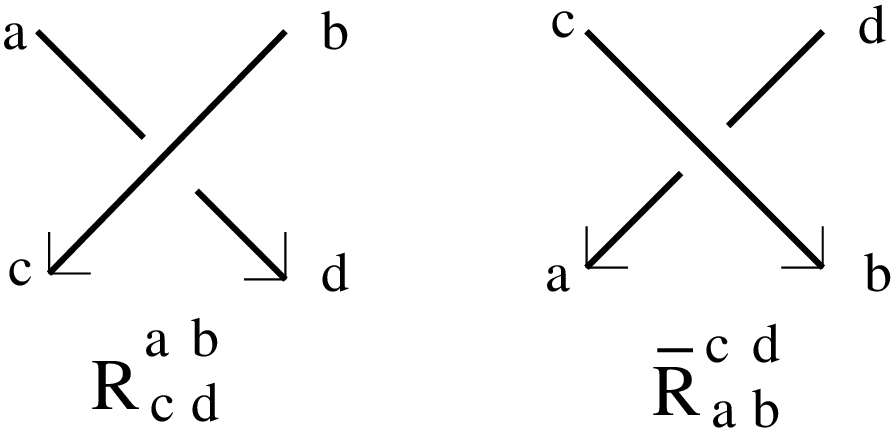}}
\caption{Boltzmann weights $R^{a,b}_{c,d}$ and $\bar R_{a,b}^{c,d}$ for positive and negative crossings}
\label{Boltzman-weightWW}
\end{figure}

%\centerline{\psfig{figure=Boltzman-weightWW.eps,height=3.3cm}}\ \\ \ \\
%\centerline{Figure 3.1; Boltzmann weights $R^{a,b}_{c,d}$ and $\bar R_{a,b}^{c,d}$ for positive and negative crossings}
\ \\
We can now generalize the number of colorings to state sum (basic notion of statistical physics)
by multiplying Boltzmann weight over all crossings and adding over all colorings:
\begin{equation*}
\mathrm{col}_{(X;BW)}(X)= \sum_{\phi\in \mathrm{col}_X(D)}\prod_{p\in \{crossings\}}\hat R^{a,b}_{c,d}(p) 
\end{equation*}
where $\hat R^{a,b}_{c,d} $ is $R^{a,b}_{c,d}$ or $\bar R^{a,b}_{c,d}$ depending on
whether $p$ is a positive or negative crossing, see \cite{Jon3,Tur} (from my personal perspective I would acknowledge the influence of Lou Kauffman paper \cite{Kau} which opened my eyes to relation of knot theory to statistical mechanics).
Our state sum is an invariant of a diagram but to get a link invariant we should test it on Reidemeister moves.
To get analogue of a shelf invariant we start from the third Reidemeister move with all positive crossings, see Figure \ref{Y-B-R3}.\footnote{In the case of a shelf (right self-distributive system) $(X,*)$ we define the operator\\ $R:X\times X \to X\times X$ by setting $R(a,b)=(b,a\ast b)$, so Boltzmann weights are $R^{a,b}_{c,d}=1$\\  when $c=b$ and $d=a\ast b$, and $0$ otherwise. For a rack we also have $\bar R(c,d)=(d\bar \ast c,c)$. }

Thus, in a general case, passing through a positive crossing is
coded by a linear map $R:kX\otimes kX\rightarrow kX\otimes kX$ which in
basis $X\times X$ is given by an $m\times m$ matrix $(R_{c,d}^{a,b})$, where $m=\left\vert X\right\vert ^{2}$. The third Reidemeister move yields the following identity between
maps from $kX\otimes kX\otimes kX$ to itself :
\begin{equation*}
(R\otimes Id)(Id\otimes R)(R\otimes Id) = (Id\otimes R)(R\otimes Id)Id\otimes R)
\end{equation*}
with graphical representation shown in Figure~\ref{Y-B-R3}. 
The above identity is called the \emph{Yang-Baxter equation}\footnote{Older names include: the star-triangle relation, the triangle equation, and the factorization equation, \cite{Jimb}.} with $R$ being called a \emph{pre-Yang-Baxter operator.}
If $R$ is additionally invertible it is called a Yang-Baxter operator. If entries of $R^{-1}$ are
equal to $\bar R^{a,b}_{c,d}$ then the state sum is invariant under ``parallel" (directly oriented)
 second Reidemeister move, see Figure \ref{Y-Binv-R2}.\footnote{We should stress that to find link invariants it suffices to use
directly oriented second and third Reidemeister moves in addition to both first Reidemeister moves,
as we can restrict ourselves to braids and use the Markov theorem. This point of view was used in \cite{Tur}.}

\begin{figure}
\centering
\scalebox{.25}{\includegraphics{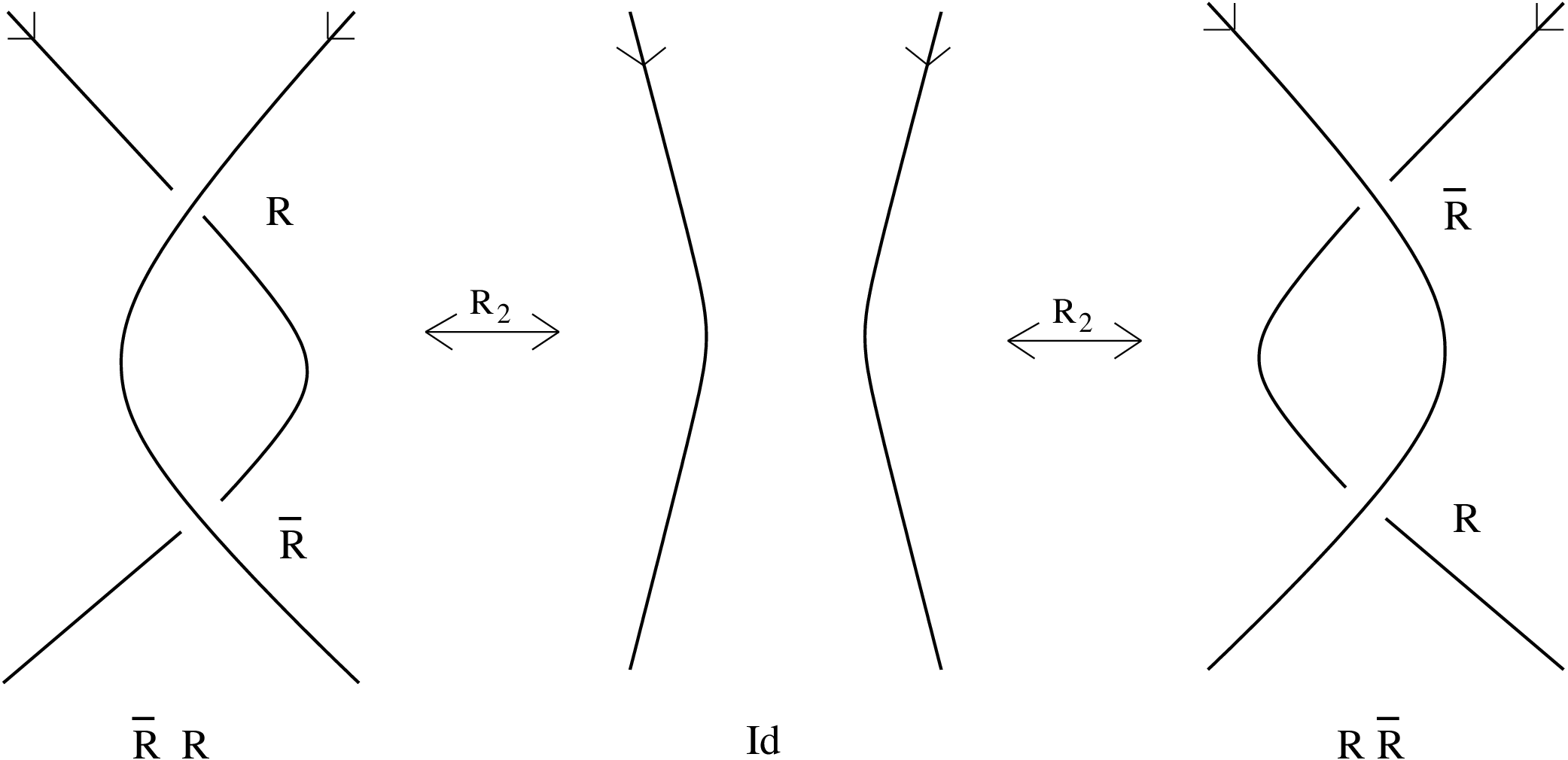}}
\caption{Invertibility of $R$ and the parallel second Reidemeister move}
\label{Y-Binv-R2}
\end{figure}

\begin{figure}[ht]
\centering
\scalebox{.50}{\includegraphics{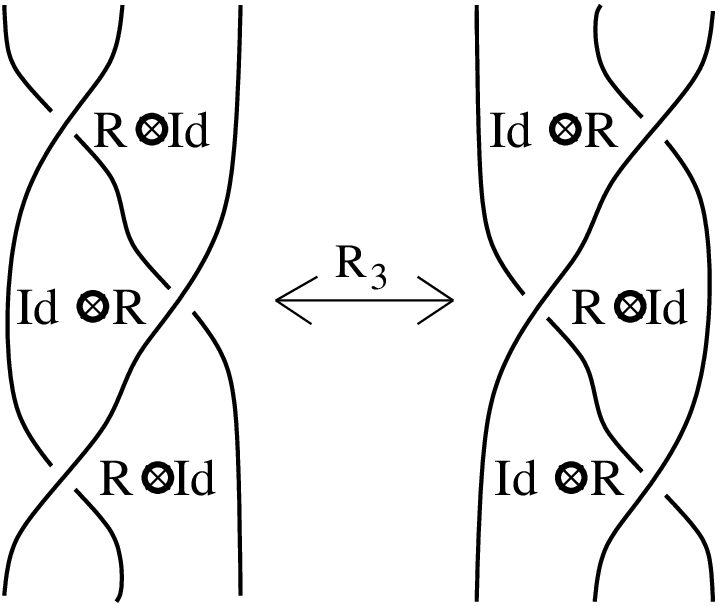}}
\caption{Yang-Baxter equation from the positive third Reidemeister move}
\label{Y-B-R3}
\end{figure}

For a given pre-Yang-Baxter operator we attempt to find presimplicial module, from which homology
will be derived.

Figure \ref{Y-Bpresim} below illustrate various graphical interpretation of the generating morphism $d_i$ of the
presimplicial category. They are related to homology of a set-theoretic
Yang-Baxter equation of Carter-Elhamdadi-Saito \cite{CES} and Fenn \cite{FIKM},
and to homology of Yang-Baxter equation of Eisermann
\cite{Eis-1,Eis-2}. We should also acknowledge stimulating observations by Ivan Dynnikov who was in audience of my Moscow talks and asked sharp questions.\footnote{Lomonosov Moscow State University, Russia; topology mini-course:\\
 (i) From knot theory to associative and distributive homology. I (May 29, 2012);\\
(ii) From knot theory to associative and distributive homology. II (May 30, 2012);\\
It was at (and after) this talk when we discussed with Ivan various visualizations of Yang-Baxter homology. Thus 22 years passed from introduction of rack homology to construction of general Yang-Baxter homology (compare \cite{Fenn}). \\
(iii) From knot theory to associative and distributive homology. III (June 1, 2012).}

\begin{figure}[ht]
\centering
\scalebox{.40}{\includegraphics{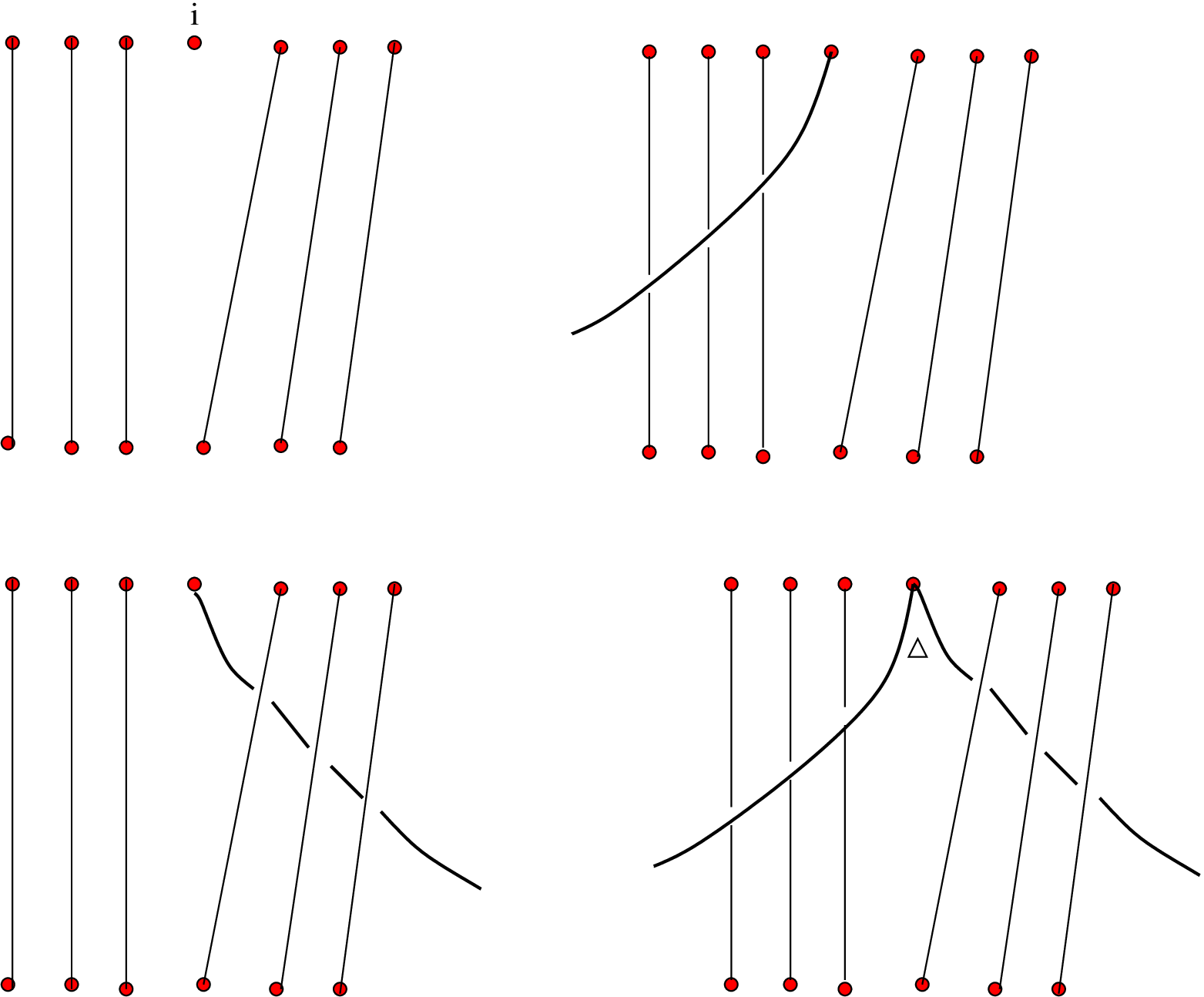}}
\caption{Various possible interpretations of the graphical face map $d_i$; one may find many more depending on needs; for 1-term Yang-Baxter homology, $d_i$ drawn in the bottom right corner was one of the first to be considered}
\label{Y-Bpresim}
\end{figure}
%\centerline{\psfig{figure=Y-Bpresim.eps,height=6.3cm}}
%\centerline{Figure 2.4; Various interpretation of the graphical face map $d_i$}

\subsection{Graphical visualization of Yang-Baxter face maps}\label{Subsection 3.1}
The presimplicial set corresponding to (two term) Yang-Baxter
homology has its graphical representation shown in Figure~\ref{Y-B-cubic}. In the case  of a set-theoretic Yang-Baxter equation we recover the homology defined in \cite{CES}; see \cite{PrWa1}.

\begin{figure}[ht]
\centering
\scalebox{.39}{\includegraphics{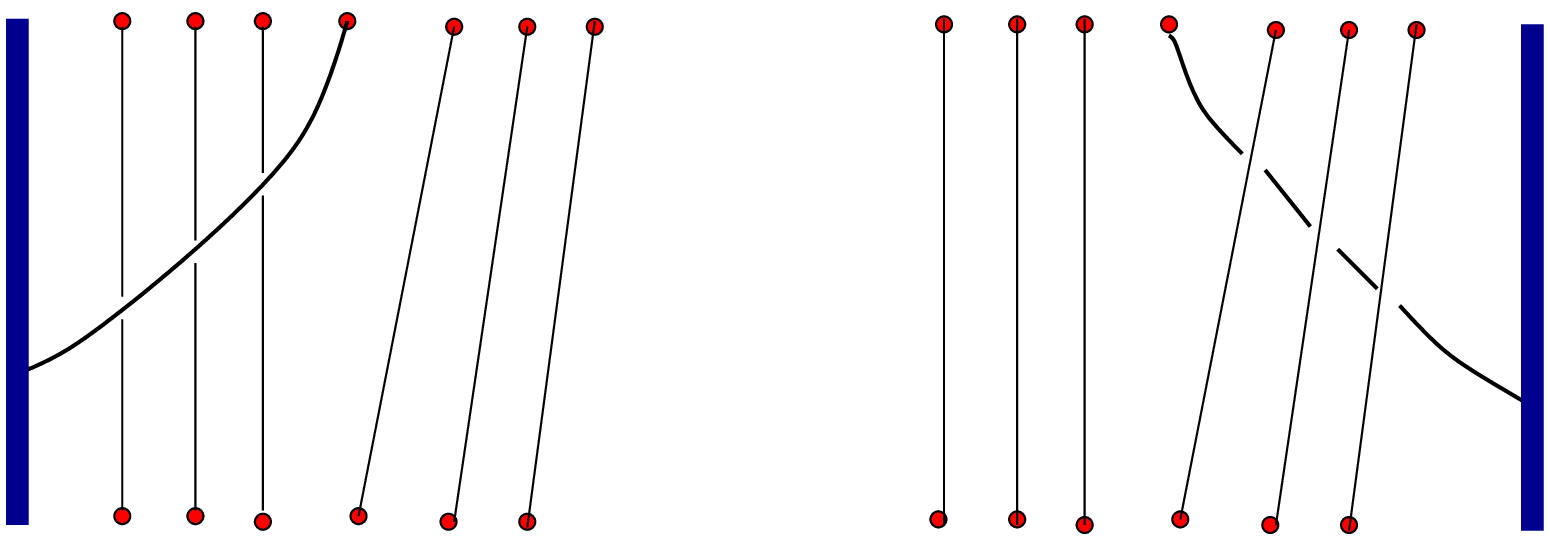}}
\caption{Graphical interpretation of the face map $d_i=d_i^{\ell}-d_i^r$}
\label{Y-B-cubic}
\end{figure}
%\centerline{\psfig{figure=di-YangBaxter.eps,height=6.3cm}}
%\centerline{Figure 2.5; Graphical interpretation of the face map $d_i$}
Our graphical model allows quite straightforward calculation as it
can be seen in the following example.
\begin{example}\label{Example 5.1}
Assume $R:X\times X \to X\times X$ generates set-theoretic Yang-Baxter operator with $R(x,y)= (R_1(x,y),R_2(x,y))$. Then
\begin{equation*}
\partial^{YB}(x_1,x_2,x_3,x_4)=\partial^{\ell} -\partial^{r},
\end{equation*}
where
\begin{eqnarray*}
\partial^{\ell}(x_1,x_2,x_3,x_4) &=& ((x_2,x_3,x_4) - (R_2(x_1,x_2),x_3,x_4)  + (R_2(x_1,R_1(x_2,x_3),R_2(x_2,x_3),x_4)\\
&-& (R_2(x_1,R_1(x_2,R_1(x_3,x_4),R_2(x_2,R_1(x_3,x_4),R_2(x_3,x_4))
\end{eqnarray*}
and
\begin{eqnarray*}
\partial^{r}(x_1,x_2,x_3,x_4)&=& (R_1(x_1,x_2),R_1(R_2(x_1,x_2),x_3),R_1(R_2(R_2(x_1,x_2),x_3),x_4)\\
&-& (x_1,R_1(x_2,x_3),R_1(R_2(x_2,x_3),x_4) + (x_1,x_2,R_1(x_3,x_4)) - (x_1,x_2,x_3).
\end{eqnarray*}
We have generally  for any $n$: 
\begin{equation*}
\partial^{\ell}_n = \sum_{i=1}^n(-1)^{i-1} d_i^{\ell}
\end{equation*}
with
\begin{eqnarray*}
&&d_i^{\ell}(x_1,\ldots,x_n)\\
&=&(R_2(x_1,R_1(x_2,R_1(x_3,\ldots,R_1(x_{i-1},x_i)))),\ldots,R_1(x_{i-1},x_i)),x_{i+1},\ldots,x_n).
\end{eqnarray*}
Similarly we have, directly from Figure~\ref{Y-B-cubic}, for any $n$:
\begin{equation*}
\partial^{r}_n =\sum_{i=1}^n(-1)^{i-1} d_i^{r}
\end{equation*}
with
\begin{eqnarray*}
&&d_i^{r}(x_1,\ldots,x_n) \\
&&= (x_1,\ldots,x_{i-1},R_1(x_i,x_{i+1}),\ldots,R_1(R_2(R_2(\ldots (R_2(x_i,x_{i+1}),x_{i+2}),\ldots,x_{n-1})x_n)))).
\end{eqnarray*}
\end{example}

\section{Decomposition of the third Reidemeister move into cubic face maps}\label{Section 6}

The main idea is illustrated in Figure~\ref{Y-B-2-cycle}. This innocently looking diagram puzzled me from Spring of 2015. I discovered it in a train from Gda\'nsk to Pozna\'n. In Pozna\'n I gave a series of talks for graduate students organized by Krzysztof Pawa{\l}owski (10 double talks) at University of Poznan May 25-29, 2015,
 which I still hope to publish (Wojtek Politarczyk made very good notes)\footnote{Title: Adventures of Knot Theorist: From Fox 3-colorings to Yang-Baxter homology with the Jones polynomial and the Khovanov homology in a background (Przygody badacza w{\c e}z{\l}\'ow: Od 3-kolorowania Foxa do homologii operatora Yanga-Baxtera z wielomianem Jonesa i homologiami Khovanova w tle).} and I always planned to apply it but never had time. Maybe one of the readers could look at it.

%\centerline{\psfig{figure=Y-B-2-cycle.eps,height=8.3cm}}

\begin{figure}
\centering
\scalebox{.35}{\includegraphics{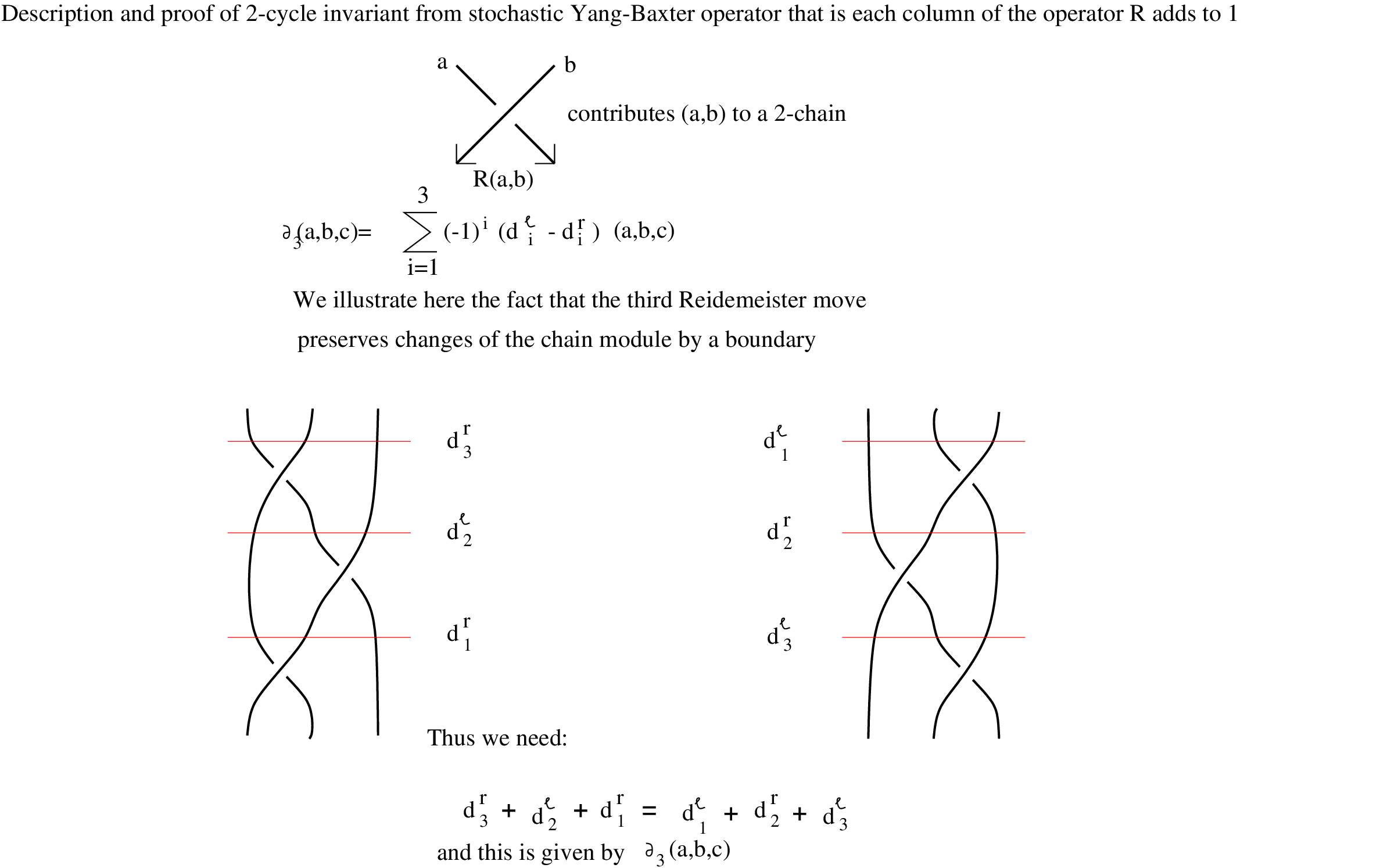}}
\caption{Reidemeister third move and face maps $d^{\varepsilon}_{i}$}
\label{Y-B-2-cycle}
\end{figure}

\section{Computation of homology for Yang-Baxter operators}\label{Section 5}

%\subsection{Homology of set theoretic Yang-Baxter operators}
%\subsection{Homology of associative versus distributive structures}

%\section{From rack homology to homology of Yang-Baxter operators}
After the long preparation we can finally define homology of general Yang-Baxter operators.
In fact it can be done in 5 minutes due to graphical interpretation developed by Victoria Lebed and myself in 2012 \cite{Leb-1,Prz10}. It is defined by the same Figure \ref{Y-B-cubic}. The difference (or rather generalization) is that we use at every crossing the general Yang-Baxter operator. Thus for set theoretic Yang-Baxter operator we construct precubic set but in general case we deal with precubic module. We decided not to discuss boundary (wall) condition and we prefer to consider only the case when the wall absorbs elements. This corresponds to the case when Yang-Baxter matrix $R$ is column unital, that is, entries of each column of $R$ add to $1$, (e.g., column stochastic matrices). Otherwise definition of the boundary operation has the same precubic form:
\begin{equation*}
\partial_n= \partial_n^{\ell} -\partial_n^r =\sum_{i=1}^n(-1)^id_i^{\ell} - \sum_{i=1}^n(-1)^id_i^{r}.   
\end{equation*}
For Yang-Baxter operators giving HOMFLYPT polynomial, only two
weights $R_{a,b}^{c,d}$ can be nonzero for a fixed $(a,b)$, i.e., $(c,d)=(a,b)$ -- smoothing case, and $(c,d)=(b,a)$ -- crossing case. Furthermore, our column unital condition yields a
relation $R_{a,b}^{a,b}+R_{b,a}^{a,b}=1$ between entries of $R$, (see
Subsection~\ref{YB-H} for details).

\subsection{Yang–Baxter homology from the quantum topology viewpoint}
I wrote proposal for Oberwolfach (Research Fellow in Oberwolfach: May 28 - June 17, 2023 (with Misha Khovanov, Louis-Adrian Roberts and Marithania Silvero) where I explained my ideas and hopes related to possible connection between Yang-Baxter homology and  Khovanov homology. I recall it in this Subsection.\footnote{Eventually we worked on a different topic (see \cite{KPRS}).}

{\it We propose to develop connections between Yang–Baxter homology, which is a generalization of the distributive homology, including the rack or quandle homology, and Khovanov homology. We envision the connection via the
cocycle invariants of links obtained from quandle or biquandle homology,
motivating it via Knot Theory.

For a non-specialist we would describe the project as follows: Mathematics that we study in school is usually associative and up till now most of the modern mathematics assumed associativity.

Khovanov homology has an associative and co-associative underlying algebraic structure. Khovanov homology was constructed for and motivated
by theory of knots and links.
Another algebraic structures discovered recently, quandles and racks, are
distributive but not associative. Just like Khovanov homology, quandles
and their homology were motivated by knot theory. No direct link between
Khovanov homology and quandle homology is currently known.
We propose a program which may link these two modern concepts. The
tools we envision are Yang–Baxter operators, powerful tools used in statistical physics, related to at least two Nobel prizes (Yang 1957, Onsager 1968). Homology of distributive magmas, which are generalization of quandles, was developed in the last ten years. From this one starts building homology of Yang–Baxter operators. This theory is still in it early phase, but several special cases have knot theory interpretations. Success in this project will bridge two fundamental notions in mathematics and physics.

More technically, the path we plan to take is as follows: Khovanov homology is the categorification of the Jones polynomial. Jones polynomial can be obtained from a specific Yang–Baxter operator.
Yang–Baxter equation generalizes distributivity in an important way. Ten
years ago homology of distributive structures were constructed and homology
of general Yang–Baxter operators was proposed. Recently, one of us jointly
with collaborators has proposed cocycle knot invariants from Yang–Baxter
operators \cite{PVY}.

We think that knot theory will help relate these two theories. 
Diagrammatic visualization of the Yang–Baxter homology given in the language of
precubic sets (so called curtain models) will be studied and applied during our work, \cite{Prz12}.}

\subsection{Yang-Baxter operators yelding Jones and HOMFLYPT polynomials}\label{YB-H} \ 
As before, let ${k}$ be a commutative ring and let $V={k}X$ be a free ${k}$-module over the basis $X=\left\lbrace  v_{1},\, v_{2}, \dots,\, v_{m} \right\rbrace$ with the ordering $ v_{a}\leq  v_{b}$ if and only if $a \leq b$. Recall, that a ${k}$-linear map 
\begin{equation*}
R: V \otimes V \longrightarrow V\otimes V
\end{equation*}
is a Yang-Baxter operator if it satisfies the Yang-Baxter equation and it is invertible. Jones discovered, \cite{Jon3}, that the Yang-Baxter operator on level $m$, given by the formula below, leads to the HOMFLYPT polynomial (formally a proper substitution in the Jones polynomial):
\begin{equation*}
R_{c \ d}^{a \ b}=
\begin{cases}
q        & \text{if $a=b=c=d;$} \\
1        & \text{if $d=a\neq b=c;$} \\
q-q^{-1}  & \text{if $c=a < b=d;$}\\
0         & \text{otherwise.}
\end{cases}    
\end{equation*}
We check easily that $R^{-1}$ is given by
\begin{equation*}
(R^{-1})_{c \ d}^{a \ b}=
\begin{cases}
q^{-1}        & \text{if $a=b=c=d;$} \\
1        & \text{if $d=a\neq b=c;$} \\
q^{-1}-q  & \text{if $c=a > b=d;$}\\
0         & \text{otherwise.}
\end{cases}
\end{equation*}
To see relation with the Jones polynomial we observe that the minimal polynomial of $R$ is quadratic: 
\begin{equation*}
R-R^{-1}=(q-q^{-1})\mathrm{Id}_{V\otimes V}.  
\end{equation*}
With Xiao Wang (my PhD student at GWU) we adjusted, in \cite{PrWa2}, the matrix above to be a column unital matrix and showed that for each $m$, $R_{(m)}$ is also a Yang-Baxter operator. Precisely, the matrix  is obtained by adding the elements in each column and dividing every element of the column by this sum. In the next theorem, we set 
\begin{equation*}
y^{2}=\dfrac{1}{1+q-q^{-1}}
\end{equation*}
\begin{theorem}[\cite{PrWa2}]\label{PW}
Let ${k}=\mathbb{Z}[y^{\pm1}]$, $m$ be a positive integer, and $V_{m}$ be the free ${k}$-module generated by the set $X_{m} = \left\lbrace v_{1},\, v_{2}, \dots,\, v_{m} \right\rbrace $ with the ordering $ v_{a}\leq  v_{b}$ if and only if $a \leq b$. Then the ${k}$-linear operator $R_{(m)}: V_{m}\otimes V_{m} \longrightarrow V_{m}\otimes V_{m}$ given by the following coefficients $R_{c \ d}^{a \ b}$, is a Yang-Baxter operator for each $m\geq 1$:
\begin{equation*}
R_{c \ d}^{a \ b}=
\begin{cases}
1       & \text{if $d=a \geq b=c;$} \\
y^{2}        & \text{if $d=a < b=c;$} \\
1-y^{2}  & \text{if $c=a < b=d;$}\\
0         & \text{otherwise.}
\end{cases}   
\end{equation*}
\end{theorem}

Directly, we see that the inverse of the operator is given by the coefficients:
\begin{equation*}
(R^{-1})_{c \ d}^{a \ b}=
\begin{cases}
1       & \text{if $d=a \leq b=c$;} \\
y^{-2}        & \text{if $d=a > b=c$;} \\
1-y^{-2}  & \text{if $c=a > b=d$;}\\
0         & \text{otherwise.}
\end{cases}  
\end{equation*}
Furthermore $R$ satisfies a quadratic relation:
\begin{equation*}
y^{-1}R-yR^{-1}=(y^{-1}-y)\mathrm{Id}_{V\otimes V}.   
\end{equation*}
The goal now is to compute homology of $R_{(m)}$. Recall that the chain modules $C_{n}=kX^n$ and the boundary homomorphisms $\partial_{n}:C_n\to C_{n-1} $ follows precubic module construction, that is,
\begin{equation*}
\partial_n=\sum_{i=1}^n(-1)^i(d_i^{\ell}-d_i^r) 
\end{equation*}
and $d_i^{\ell}$ and $d_i^r$, have the graphical interpretation in Figure \ref{Y-B-cubic}. In the next subsection we survey actual state of knowledge, mostly the work with my student Xiao Wang \cite{PrWa2}.

\subsection{Homology of HOMFLYPT Yang-Baxter operators} 
The first nontrivial general result on homology of Yang-Baxter operator which 
is not set-theoretic Yang-Baxter operator is as follows. 
\begin{theorem}[\cite{PrWa2}]\label{PW2}
Let $R_{(m)}$ be a unital Yang-Baxter operator giving the Homflypt polynomial on level $m$ as in Theorem \ref{PW}. Then
\begin{equation*}
H_{2}(R_{(m)})={k}^{1+ \binom{m}{2}} \oplus \left( \dfrac{{k}}{1-y^{2}} \right)^{\binom{m}{2}} \oplus  \left( \dfrac{{k}}{1-y^{4}} \right)^{m-1}.
\end{equation*}
In particular, the ring ${k}$ can be either $\mathbb{Z}[y^{\pm 1}]$ or $\mathbb{Z}[y]$.
\end{theorem}

%\begin{theorem}\cite{PrWa2}

\begin{conjecture}[\cite{PrWa2}]
\begin{equation*}
H_3(R_{(m)})=k^{m(8-3m+m^2)/6} \oplus (k/(1-y^2))^{(m^2-1)(5m-6)} \oplus (k/(1-y^4))^{m-1}.
\end{equation*}
Notice that for $m=3$ we have 
$\frac{m(8-3m+m^2)}{6}=4=(m-1)^2$.
\end{conjecture}

{\color{red} Added for e-print:\\ We solved the conjecture at our Mathathon 9 in December 2024. \footnote{We computed also $H_4(R_{(m)})$ to get \\ $k^{\frac{m^4 - 6m^3 + 23m^2 - 18m + 24}{24}}\oplus (k/(1-y^2))^(\frac{(m-1)(23m^{3}-3m^{2}-26m+24)}{24} \oplus (k/(1-y^4))^{\frac{(m-1)(m^2+m+2)}{2}}$. }}

\begin{conjecture}\cite{PrWa3}
The rank of the free part of $H_n(R_{(m)})$ 
is equal to $2^{m-1}$ for $n\geq m-1$, and
\begin{equation*}
\mathrm{rank} (H_n(R_{(m)})) = \binom{m-1}{0} + \binom{m-1}{1} +\ldots + \binom{m-1}{n},  
\end{equation*}
for $0<n\leq m-1$.
\end{conjecture}

{\color{red} Added for e-print:\\ We proved that the right side of the equality is in fact the lower bound for $\mathrm{rank} (H_n(R_{(m)}))$.} 

\begin{conjecture}[\cite{PrWa1}]\label{HnR2}
\begin{equation*}
H_n(R_{(2)})= k^2\oplus (k/(1-y^2))^{a_n}\oplus  (k/(1-y^2))^{s_{n-2}},   
\end{equation*}
where 
\begin{equation*}
s_{n-2}=\sum_{i=1}^{n-1}f_i = f_{n+1}-1,\,\,\text{and}\,\, 2^n=2+a_{n-1}+s_{n-3}+a_n+s_{n-2}. 
\end{equation*}
\end{conjecture}
Notice that 
\begin{equation*}
2 + s_{n-3}+s_{n-2} = f_n+f_{n+1} = f_{n+2}\,\,\text{with}\,\, a_1=0.
\end{equation*}
Thus, $a_n= 2^n - a_{n-1}-f_{n+2}$.

We can write a short closed formula for $a_n$, and we notice that it is analogous to calculation in Subsection \ref{nmovemdegree} for $m=1$. 
Thus we have:
\begin{eqnarray*}
a_n &\stackrel{(1)}{=}& 2^n - a_{n-1}-f_{n+2} \stackrel{(1)}{=}2^n-2^{n-1}+a_{n-2}-f_{n+2}+f_{n+1}  \stackrel{(3)}{=} \ldots \\
&\stackrel{(n)}{=}& (2^n-2^{n-1} + \ldots + (-1)^n) - (f_{n+2}-f_{n+1} + \ldots +(-1)^nf_2) \\
&=& \frac{2^{n+1}+(-1)^n}{3} - f_{n+1}.
\end{eqnarray*}
In the proof we use the standard Fibonacci numbers formula:
\begin{equation*}
\sum_{i=2}^m(-1)^{m-i}f_i= f_{m-1}.    
\end{equation*}
\begin{example}
Consider the Yang-Baxter operator $R_{(m)}$ for arbitrary $m\geq 2$. 
We will graphically compute part of $\partial(a,b,c,d)$ in the case of 
$a\leq b\leq c\leq d$. Generally we have $\partial_4(a,b,c,d)=\sum_{i=1}^4(-1)^i (d_i^{\ell}-d_i^r) .$
We illustrate calculation of $d_4^{\ell}(a,b,c,d)$ see Figure \ref{D4-compTree}. 
From the figure we get 
\begin{eqnarray*}
d_4^{\ell}(a,b,c,d)&=& y^6(a,b,c) +y^4(1-y^2)(d,b,c)+ y^4(1-y^2)(a,d,c) + y^2(1-y^2)^2(b,d,c)\\
&+& y^4(1-y^2)(a,b,d)+ y^2(1-y^2)^2(c,b,d) + y^2(1-y^2)^2(a,c,d)\\
&+&(1-y^2)(b,c,d).
\end{eqnarray*}
\end{example}
In February of 2020, we computed with Masahico Saito and Mohamed Elhamdadi the boundary $\partial_4(C_4)$ with the goal to find the third homology of $R_{(m)}$ but Covid stopped our efforts. My plan is to devote to this problem Mathathon 9 which will take place in December 2024 at GWU. I hope to report our funding at the next B{\c e}dlewo conference, in Summer 2025.\\
{\color{red}Added for e-print: We computed, not only $H_3(R_{(m)})$ and $H_4(R_{(m)})$ but also, partially $H_5(R_{(m)})$ and $H_6(R_{(m)})$ discovering, in particular a new torsion in homology $H_6$; see \cite{CCGOPWY}. }

%\newpage
\begin{figure}[ht]
\centering
\scalebox{.35}{\includegraphics{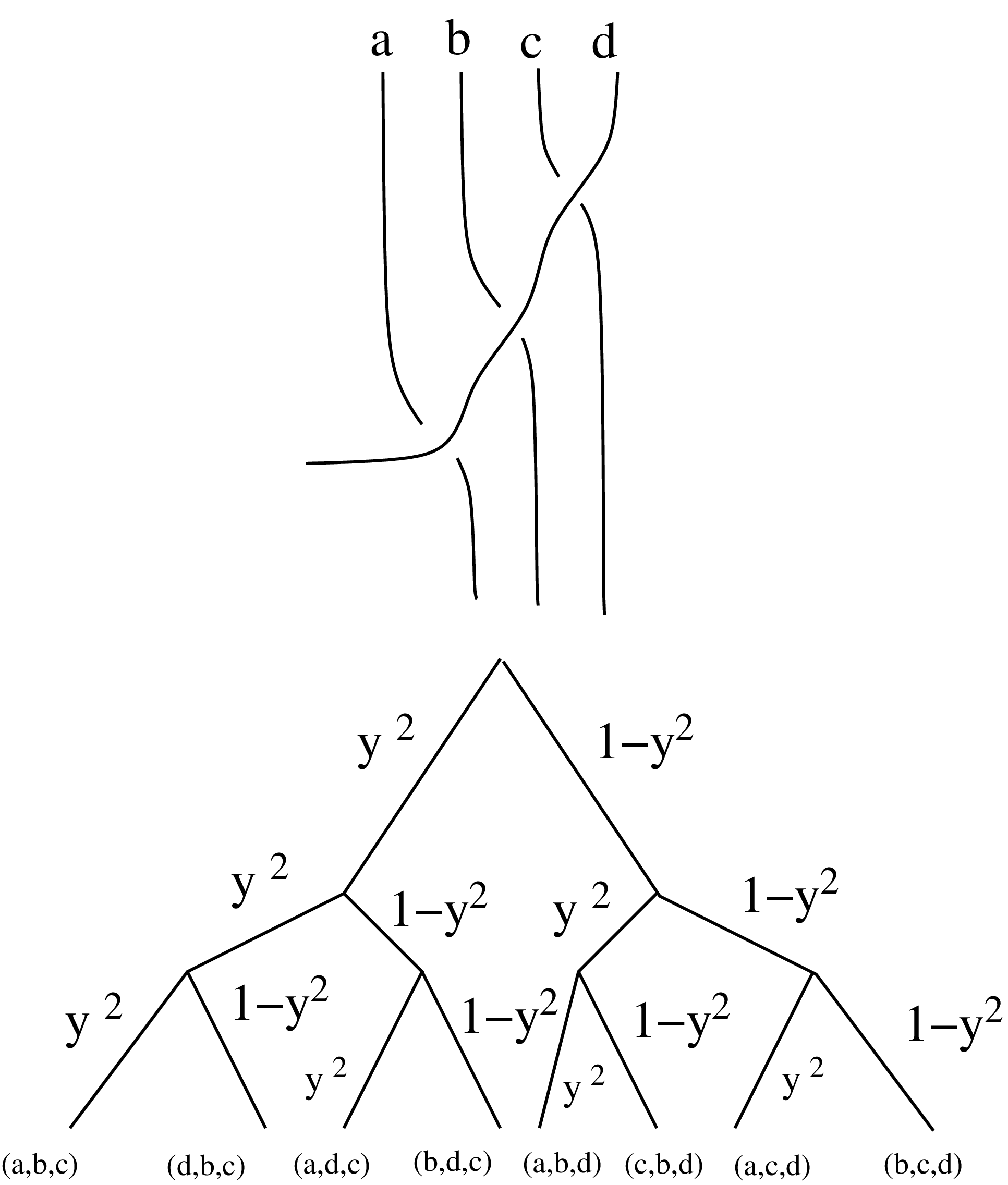}}
%\scalebox{.40}{\includegraphics{Y-Bpresim.eps}}
\caption{Computational tree for $d_4^{\ell}(a,b,c,d)$ where $a\leq b \leq c \leq d $. We sum over all leaves with coefficients being the product of weights of edges from the root to the given leaf}
\label{D4-compTree}
\end{figure}

\newpage

\end{document}